\documentclass[12pt,reqno]{amsart}

\usepackage[utf8]{inputenc}

\usepackage{amsmath} %General package for maths (loads also amsbsy to make bold symbols)
\usepackage{amsthm} %Package for theorems
\usepackage{amssymb} %Symbols (loads also amsfonts)
\usepackage{amscd} %Package for rectangular diagrams
\usepackage{enumitem} %Enumerations
\usepackage{mathtools} %More symbols, etc.
\usepackage{mathrsfs} %Calligraphic symbols with \mathscr
\usepackage{esint} %To write averaged integrals

\usepackage{graphicx} %For using \includegraphics
\usepackage{subcaption}
\usepackage{pgf,tikz,pgfplots}
\usetikzlibrary{arrows}
\usetikzlibrary{intersections}
\usetikzlibrary{fadings}

\usepackage[colorlinks]{hyperref} %Makes references hyperlinks
\hypersetup{
	linktocpage = true,
	colorlinks=true,
	linkcolor=[rgb]{0,0.5,0.5},
	citecolor=blue,
}
\usepackage{cleveref,autonum}
\usepackage[colorinlistoftodos]{todonotes} %Package to add TODO notes in colors
\newtheorem{theorem}{Theorem}[section]
\newtheorem*{theorem*}{Theorem}
\newtheorem{proposition}[theorem]{Proposition}
\newtheorem{lemma}[theorem]{Lemma}
\newtheorem{corollary}[theorem]{Corollary}

\theoremstyle{definition}

\theoremstyle{remark}
\newtheorem{remark}[theorem]{Remark}
\newtheorem*{remark*}{Remark}

\newcommand{\con}[1]{\mathbb{#1}}
\newcommand{\R}{\con{R}} %Real
\newcommand{\Z}{\con{Z}} %Integer
\newcommand{\Sph}{\con{S}} %Sphere

\newcommand{\acal}{\mathcal{A}}

\usepackage{euscript}
\usepackage{relsize}
\DeclareMathAlphabet{\mathpzc}{OT1}{pzc}{m}{it}
\DeclareMathAlphabet\euscr{T1}{qzc}{m}{n}
\newcommand{\B}{\mathchoice
	{\mathlarger{\euscr{B}}}
	{\ds \mathlarger{\euscr{B}}}
	{\scalebox{0.85}{\ensuremath {\euscr{B}}}}
	{\euscr{B}}
}

\makeatletter
\newcommand{\leqnomode}{\tagsleft@true\let\veqno\@@leqno}
\newcommand{\reqnomode}{\tagsleft@false\let\veqno\@@eqno}
\makeatother

\newcommand{\abs}[1]{\left| #1 \right|}
\newcommand{\norm}[1]{\left \| {#1} \right \| }
\newcommand{\seminorm}[1]{\left [ {#1} \right ] }

\newcommand{\ep}{\varepsilon}
\newcommand{\s}{s}
\newcommand{\fraclaplacian}{(-\Delta)^\s}
\newcommand{\halflaplacian}{(-\Delta)^{1/2}}

\newcommand{\Lip}{\mathrm{Lip}}

\newcommand{\loc}{\mathrm{loc}}
\renewcommand{\d}{\,\mathrm{d}} %straight d with small space before
\newcommand{\df}{\mathrm{d} } %straight d without space 
\newcommand{\dx}{\,\mathrm{d}x} %The usual differential of x
\newcommand{\bpar}[1]{\left ( {#1}\right )}
\newcommand{\setcond}[2]{\left \{ #1 \ : \ #2  \right \}}

\newcommand\beqc[1]{\left\{\begin{array}{#1}}
	\newcommand\eeqc{\end{array} \right.}
\def\PDEsystem{rcll}
\def\bmatrix{\begin{pmatrix}}
	\def\ematrix{\end{pmatrix}}

\DeclareMathOperator{\dist}{dist}

\let\div\relax
\DeclareMathOperator{\div}{div}

\def\ds{\displaystyle}

\usepackage[textheight=22cm,textwidth=16.6cm,headsep=1cm,centering]{geometry}

\usepackage{setspace}

\numberwithin{equation}{section}

\graphicspath {{Images/}}

\title[Regularity of stable solutions up to dimension 4]{A universal Hölder estimate up to dimension 4 for stable solutions to half-Laplacian semilinear equations}
\author{Xavier Cabr\'e}
\address{X. Cabr\'e \textsuperscript{1,2}
	\newline
	\textsuperscript{1} ICREA, Pg. Lluis Companys 23, 08010 Barcelona, Spain
	\newline
	\textsuperscript{2} Universitat Polit\`ecnica de Catalunya, Departament de Matem\`{a}tiques and IMTech, 
	Av. Diagonal 647, 08028 Barcelona, Spain}
\email{xavier.cabre@upc.edu}
\author{Tomás Sanz-Perela}
\address{T. Sanz-Perela:
	Departamento de Matemáticas, Universidad Autónoma de Madrid, Ciudad Universitaria de Cantoblanco, 28049 Madrid, Spain}
\email{tomas.sanz@uam.es}

\thanks{Both authors are supported by grants MTM2017-84214-C2-1-P and RED2018-102650-T funded by MCIN/AEI/10.13039/501100011033 and by ``ERDF A way of making Europe''. They are members of the Barcelona Graduate School of Mathematics (BGSMath) and of the Catalan research group 2017 SGR 01392. The second author acknowledges financial support from the Spanish Ministry of Economy and Competitiveness (MINECO), through the María de Maeztu Program for Units of Excellence in R\&D MDM-2014-0445, as well as from the EPSRC grant EP/S03157X/1.}


\begin{document}
	
\setstretch{1.0740}

%%%%%%%%%%%%%%%%%%%%%%%%%%%%%%%%%%%%%%%%%%%%%%%%%%%%%%%%%%%%%%%%%%%%%%%%%%%%%%%
%%%%%%%%%%%%%%%%%%%%%%%%%%%%%%%%%%%%%%%%%%%%%%%%%%%%%%%%%%%%%%%%%%%%%%%%%%%%%%%
\begin{abstract}
	
	We study stable solutions to the equation $(-\Delta)^{1/2} u = f(u)$, posed in a bounded domain of $\mathbb{R}^n$.
	For nonnegative convex nonlinearities, we prove that stable solutions are smooth in dimensions $n\leq 4$. 
	This result, which was known only for $n=1$, follows from a new interior H\"older estimate that is completely independent of the nonlinearity~$f$.
	
	A main ingredient in our proof is a new geometric form of the stability condition. 
	It is still unknown for other fractions of the Laplacian and, surprisingly, it requires convexity of the nonlinearity. 
	From it, we deduce higher order Sobolev estimates that allow us to extend the techniques developed by Cabr\'e, Figalli, Ros-Oton, and Serra for the Laplacian.
	In this way we obtain, besides the Hölder bound for $n\leq 4$, a universal $H^{1/2}$ estimate in all dimensions.
	
	Our $L^\infty$ bound is expected to hold for $n\leq 8$, but this has been settled only in the radial case or when $f(u) = \lambda e^u$. 
	For other fractions of the Laplacian, the expected optimal dimension for boundedness of stable solutions has been reached only when $f(u) = \lambda e^u$, even in the radial case.

\end{abstract}
%%%%%%%%%%%%%%%%%%%%%%%%%%%%%%%%%%%%%%%%%%%%%%%%%%%%%%%%%%%%%%%%%%%%%%%%%%%%%%%
%%%%%%%%%%%%%%%%%%%%%%%%%%%%%%%%%%%%%%%%%%%%%%%%%%%%%%%%%%%%%%%%%%%%%%%%%%%%%%%

\hspace{-.3cm}

\maketitle

\setcounter{tocdepth}{1}

\tableofcontents

\newpage

%%%%%%%%%%%%%%%%%%%%%%%%%%%%%%%%%%%%%%%%%%%%%%%%%%%%%%%%%%%%%%%%%%%%%%%%%%%%
%%%%%%%%%%%%%%%%%%%%%%%%%%%%%%%%%%%%%%%%%%%%%%%%%%%%%%%%%%%%%%%%%%%%%%%%%%%%
\section{Introduction and results}
%%%%%%%%%%%%%%%%%%%%%%%%%%%%%%%%%%%%%%%%%%%%%%%%%%%%%%%%%%%%%%%%%%%%%%%%%%%%
%%%%%%%%%%%%%%%%%%%%%%%%%%%%%%%%%%%%%%%%%%%%%%%%%%%%%%%%%%%%%%%%%%%%%%%%%%%%

The regularity of stable solutions to semilinear  equations $-\Delta u = f(u)$ has been a long-standing problem in elliptic PDEs since the 1970s.
Important efforts have been devoted to investigate the optimal dimension up to which stable solutions are bounded.
This problem has been recently solved by Cabré, Figalli, Ros-Oton, and Serra~\cite{CabreFigalliRosSerra-Dim9}, by proving that stable solutions  are regular in dimensions $n\leq 9$  for all nonnegative nonlinearities $f$.
The result is optimal since there exist examples of singular $H^1$ stable solutions in dimensions $n \geq 10$. 
For further details, see~\cite{CabreFigalliRosSerra-Dim9, Dupaigne} and the references therein.

The goal of this paper is to study the same question  in the fractional setting.
We consider the equation 
\begin{equation}
	\label{Eq:SemilinearEquation}
	\fraclaplacian u = f(u) \quad \text{ in } \Omega,
\end{equation}
where $\Omega$ is a bounded domain of $\R^n$ and $\fraclaplacian$ is the fractional Laplacian,
\begin{equation}
	\label{Eq:DefFracLap}
	\fraclaplacian w\,(x') := c_{n,\s} \int_{\R^{n}} \dfrac{w(x') - w(z')}{|x'-z'|^{n + 2\s}}  \d z' ,\quad \s \in (0,1),
\end{equation}
with  $ c_{n,\s}$ being a positive normalizing constant (see \cite{CabreSireI}).
In this case, the few known results (mainly those contained in the three papers~\cite{RosOtonSerra-Extremal, RosOton-Gelfand, SanzPerela-Radial} described below) reach the expected optimal dimension for boundedness of stable solutions only when $f(u) = \lambda e^u$, even in the radial case.
In \Cref{Fig:KnownResults} such dimension, which was found by Ros-Oton in \cite{RosOton-Gelfand} and it is given by condition \eqref{Eq:OptimalDimensionsGammas} below, is compared with the available~results.

Note that \eqref{Eq:SemilinearEquation} is the Euler-Lagrange equation of the functional
\begin{equation}
	\label{Eq:Energy}
	E(w) := \dfrac{c_{n,\s}}{4}\int \int_{\R^{2n} \setminus (\Omega^c)^2} \dfrac{|w(x') - w(z')|^2}{|x'-z'|^{n + 2\s}} \d x' \d z' - \int_\Omega F(w) \d x',
\end{equation}
where $F(t) := \int_0^t f(\theta) \d \theta$ and $\Omega^c := \R^n\setminus \Omega$. 
A solution $u:\R^n \to \R$ to \eqref{Eq:SemilinearEquation} ---or critical point of $E$--- is said to be \emph{stable} if the second variation of $E$ at $u$ is nonnegative, i.e., $\frac{\d^2}{\d \varepsilon^2}|_{\varepsilon = 0} E(u + \varepsilon \xi) \geq 0$ for all $\xi\in H^{\s}(\R^n)$ with compact support in $\Omega$.
This is equivalent to requiring that
\begin{equation} 
	\label{Eq:Stability}
	\int_{\Omega} f'(u)\xi^2 \d x'
	\leq
	\seminorm{\xi}_{H^{\s}(\R^n)}^2 \quad \text{ for all } \ \xi\in H^{\s}(\R^n) \text{ with compact support in } \Omega. 
\end{equation}
Recall that for $\s \in (0,1)$ and $U\subset \R^n$, we define $H^\s (U) := \{ w \in L^2(U): \seminorm{w}_{H^{\s}(U)} < +\infty\}$, where 
\begin{equation}
	\label{Eq:DefHs}
	\seminorm{w}_{H^{\s}(U)}^2 := \dfrac{c_{n,\s}}{2}\int_U \int_U \dfrac{|w(x') - w(z')|^2}{|x'-z'|^{n + 2\s}} \d x' \d z'.
\end{equation}
Notice that stability is considered among functions which agree with $u$ outside $\Omega$, and therefore \emph{local minimizers} of the energy (i.e., minimizers under small perturbations which do not change the exterior values of $u$) are stable solutions.

Our interest lies in nonnegative nonlinearities that grow superlinearly at $+\infty$. 
In this case, it is easy to see that the energy \eqref{Eq:Energy} is unbounded below and hence admits no absolute minimizer.
However, as we will see in \Cref{Subsec:Extremal}, there are important instances in which nonconstant stable solutions exist.

In this article we study the problem for the half-Laplacian ($\s=1/2$), a case which is of special interest in view of its applications.
Indeed, $\halflaplacian$ and related first order integro-differential operators appear in the modeling of important physical phenomena, such as the Peierls-Nabarro model for crystal dislocations or the Benjamin-Ono equation in fluid dynamics.
This occurs since the half-Laplacian is the Dirichlet to Neumann map associated to the harmonic extension in the half-space.

An important illustration of  the key role played by the dimension in our problem is given by the function $u(x') = \log |x'|^{-1}$ ($x'\in \R^n$), which solves  $\halflaplacian u = \lambda_0 e^u$ in $B_1$ for some constant $\lambda_0 >0$ and belongs to $H^{1/2}(B_1)\cap L^1_{1/2}(\R^n)$ ---see \eqref{Eq:DefL1s} below for this last space. 
Using the fractional Hardy inequality  one can check that, in dimensions $n \geq 9$, $u$ is a stable solution.

On the other hand, for the equation $\halflaplacian u = f(u)$ in $\Omega \subset \R^n$, with $f \geq 0$ and under the Dirichlet condition $u \equiv 0$ in $\R^n \setminus \Omega$, it is known that stable solutions are bounded in $\Omega$
\begin{itemize}
	\item when $n= 1$, if $f$ is convex (Ros-Oton and Serra \cite{RosOtonSerra-Extremal});
	\item when $n \leq 4$, if $f f'' /(f')^2$ has a limit at infinity (Ros-Oton and Serra \cite{RosOtonSerra-Extremal});
	\item when $n \leq 8$, if $f(u) = \lambda e^u$ and $\Omega$ is symmetric and convex with respect to all the coordinate directions (Ros-Oton \cite{RosOton-Gelfand});
	\item when $2 \leq n \leq 8$, if $\Omega= B_1$ (Sanz-Perela \cite{SanzPerela-Radial}).
\end{itemize}

In view of these results, it is natural to conjecture that stable solutions to $\halflaplacian u = f(u)$ in $\Omega\subset \R^n$ are always bounded in $\Omega$ whenever $n\leq 8$, not only for $f(u) = \lambda e^u$ in symmetric domains but for a wider class of nonlinearities and domains.
Here the dimension $n=8$ would be optimal, by the explicit stable solution for $n\geq 9$ exhibited above.
In this paper we make progress towards the solution of this conjecture by extending some of the techniques developed for the local case in \cite{CabreFigalliRosSerra-Dim9} by Cabré, Figalli, Ros-Oton, and Serra.

The following is our main result.
It provides a universal interior Hölder estimate for stable solutions $u$ to \eqref{Eq:SemilinearEquation} in dimensions $n\leq 4$, as well as an $H^{1/2}$ bound in every dimension.
Both estimates give a control in terms of a very weak norm of $u$, namely the quantity $\norm{u}_{L^1_{1/2} (\R^n)}$.
Here and through the paper, for $\s\in (0,1)$ we denote by $L^1_{\s} (\R^n)$ the space of measurable functions for which the norm
\begin{equation}
	\label{Eq:DefL1s}
	\norm{w}_{L^1_{\s} (\R^n)} := \int_{\R^n} \dfrac{ |w(x')|}{(1 + |x'|^2)^{\frac{n + 2\s}{2}}} \d x'
\end{equation}
is finite.

\begin{theorem}
	\label{Th:Holder}
	Let $n\geq 1$ and $u\in C^2(B_1)\cap L^1_{1/2}(\R^n)$ be a stable solution to $(-\Delta)^{1/2}u = f(u)$ in $B_1 \subset \R^n$, where $f$ is a nonnegative convex $C^{1,\gamma}$ function for some $\gamma >0$. 
	
	Then,	
	\begin{equation}
		\label{Eq:H1/2}
		\seminorm{u}_{H^{1/2} (B_{1/2})} \leq C   \norm{u}_{L^1_{1/2} (\R^n)} 
	\end{equation}
	for some dimensional constant $C$.
	In addition,  
	\begin{equation}
		\label{Eq:Holder}
		\norm{u}_{C^\alpha (\overline{B}_{1/2})} \leq C   \norm{u}_{L^1_{1/2} (\R^n)}  \qquad \text{if } 1 \leq n\leq 4,
	\end{equation}
	for some dimensional constants $\alpha>0$ and $C$.
\end{theorem}

The most remarkable feature of this result is that estimates \eqref{Eq:H1/2} and  \eqref{Eq:Holder} do not depend on the nonlinearity $f$ at all, which is only assumed to be nonnegative and convex.
This is a notorious difference with all the available results mentioned above, in which the $L^\infty$ estimates depended on the particular  nonlinearity $f$ appearing in the equation. This feature will allow us to establish a Liouville result for entire stable solutions, \Cref{Coro:Liouville} below.

In addition, \Cref{Th:Holder} is the first result in which no prescribed exterior Dirichlet condition is assumed. 
This will be of great importance in order to include the case $n=1$ (by looking at the solution as defined in $\R^2$ after adding an artificial variable) since our basic inequality towards the Hölder estimate \eqref{Eq:Holder} ---which is \eqref{Eq:EstimatevrAnnulusHalfLap} below--- requires\footnote{More generally, for $\s \in (0,1]$ it requires $n > 2\s$; see \Cref{Prop:EstimatevrAnnulus} below for $\s\in (0,1)$ and \cite{CabreFigalliRosSerra-Dim9} for $\s = 1$.} $n\geq 2$.

The Hölder interior estimate of \Cref{Th:Holder} for $n\leq 4$ can be combined with the moving planes method to obtain an $L^\infty(\Omega)$ bound for stable solutions to the Dirichlet problem in convex domains $\Omega$ under zero exterior data.
Indeed, as proved in~\cite{RosOtonSerra-Extremal}, one can start the moving planes argument at points on the boundary of a convex domain $\Omega$, obtaining an $L^\infty$ estimate in a neighborhood of $\partial \Omega$.
Thus, one concludes the following result.

\begin{corollary}
	\label{Coro:LinftyConvexDiriclet}
	Let $1 \leq n \leq 4$ and let $\Omega \subset \R^n$ be any bounded convex $C^1$ domain.	
	Let $u\in L^\infty(\Omega)\cap H^{1/2}(\R^n)$ be a stable solution to
	\begin{equation}
		\label{Eq:DirichletPb}
		\beqc{\PDEsystem}
		\halflaplacian u &= & f(u) &  \text{in } \Omega ,\\
		u &= &0 &  \text{in }    \R^n\setminus\Omega,
		\eeqc
	\end{equation}
	where $f$ is a nonnegative convex $C^{1,\gamma}$ function for some $\gamma >0$.
	
	Then, 
	\begin{equation}
		\label{Eq:LinftyConvex}
		\norm{u}_{L^\infty(\Omega)} \leq C_\Omega  \norm{u}_{L^1(\Omega)} 
	\end{equation}
	for some constant $C_\Omega$ depending only on $\Omega$.
\end{corollary}

Also in convex domains $\Omega$, an $H^\s(\R^n)$ estimate for bounded stable solutions to \eqref{Eq:SemilinearEquation} vanishing outside $\Omega$ is known to hold in every dimension and for all $\s\in(0,1)$.
This was proved by Ros-Oton and Serra in \cite{RosOtonSerra-Extremal} using the Pohozaev identity for the fractional Laplacian combined with some regularity estimates near $\partial \Omega$.
Instead, our interior $H^{1/2}$ bound \eqref{Eq:H1/2} does not assume any particular exterior data ---and also holds for all dimensions.

Although \Cref{Th:Holder} and \Cref{Coro:LinftyConvexDiriclet} are stated as a priori estimates for $C^2$ or $L^\infty$ solutions, they also hold for a bigger class of stable solutions.
This is discussed in \Cref{Remark:Regularity} below, where we comment also on the different notions of solution to our semilinear equation.

As a main ingredient in the proof of \Cref{Th:Holder}, we establish a new geometric\footnote{\label{Footnote:Curvatures}Although we do not use it in this article, it is worth noticing that the quantity $\acal$ defined in \eqref{Eq:AcalDef} controls a geometric quantity: the second fundamental form of the level sets of $v$. Indeed, in the set $\{|\nabla v|>0\}$ it holds $\acal^2= |\nabla_T |\nabla v||^2 + |B|^2 |\nabla v |^2$, where $|B|^2 = |B(x)|^2$ denotes the square of the second fundamental form of the level set of $v$ passing through $x\in \R^{n+1}_+$, and $\nabla_T$ denotes the tangential gradient along such level set; see Lemma~2.1 in~\cite{SternbergZumbrun1} for a detailed proof.
} form of the stability condition, expressed through the harmonic extension\footnote{Recall that the equation $\halflaplacian u = f(u)$ is equivalent to $\partial_\nu v = f(v) $ on $\partial \R^{n+1}_+$.} $v$ of $u$ in $\R^{n+1}_+$. 
It is stated in  \eqref{Eq:GeomStabilityFullGradient_s=1/2} below, and it is a fractional analogue of a well-known inequality of Sternberg and Zumbrun \cite{SternbergZumbrun1,SternbergZumbrun2} for stable solutions to $-\Delta u = f(u)$.
Surprisingly (when comparing it with the proof in the local case),  to establish \eqref{Eq:GeomStabilityFullGradient_s=1/2} we need to further assume that $f$ is~convex.

\begin{theorem}
	\label{Th:GeomStabilityFullGradient_s=1/2}
	Let $n \geq 1$ and $u\in C^2(\Omega) \cap L^1_{1/2}(\R^n)$ be a stable solution to $(-\Delta)^{1/2}u = f(u)$ in a domain $\Omega \subset \R^n$, where $f$ is a nonnegative convex $C^{1,\gamma}$ function for some $\gamma>0$. 
	Let $v$ be the harmonic extension of $u$ in $\R^{n+1}_+$ and define
	\begin{equation}
		\label{Eq:AcalDef}
		\acal := \begin{cases}
			\left ( \ds \sum_{i,j=1}^{n+1}  v_{ij}^2 -  \sum_{j=1}^{n+1}\left( \sum_{i=1}^{n+1} v_{ij} \dfrac{v_i }{|\nabla v|}  \right)^2  \right)^{1/2} & \text{ if }\, |\nabla v| > 0, \vspace{2mm} \\
			\ 0 &  \text{ if }\, |\nabla v| = 0 .
		\end{cases} 
	\end{equation}
	
	Then,  
	\begin{equation}
		\label{Eq:GeomStabilityFullGradient_s=1/2}
		\int_{\R^{n+1}_+ } \acal^2 \eta ^2 \d x \leq \int_{\R^{n+1}_+} |\nabla v|^2  |\nabla \eta |^2  \d x
	\end{equation}
	for every Lipschitz function $\eta$ with compact support in $\Omega \times [0,+\infty)$.
\end{theorem}

An analogue of \Cref{Th:GeomStabilityFullGradient_s=1/2} for fractional powers $\s \neq 1/2$ is not known at the moment. In its proof we strongly use that $\s=1/2$; see \Cref{Remark:s=1/2} for details.
We believe that, if available, it would allow extending \Cref{Th:Holder}, in certain dimensions, to all powers $\s \in (0,1)$.
%	 of the Laplacian.

Prior to our work, the inequality of Sternberg and Zumbrun had been extended to the fractional setting by Sire and Valdinoci~\cite{SireValdinoci}, but in a weaker form than ours.
Their inequality involves only horizontal derivatives of $v$: they obtained a similar quantity  to $\acal$ in \eqref{Eq:AcalDef}, but where the indices run only from $1$ to $n$, and not to $n+1$. 
As a consequence, no one has used it successfully for proving boundedness of stable solutions in bounded domains ---although it is the main ingredient in a simple proof of the fractional De Giorgi conjecture in $\R^2$; see~\cite{SireValdinoci}.

Instead, from our new geometric form of the stability condition we will obtain an $L^2$ estimate for all second derivatives of $v$, including those involving the $y$ variable.
This bound will be crucial in the proof of \Cref{Th:Holder}; see~\Cref{Subsec:IdeasOfProofs} for more details.

Before going into the main ideas in our proofs, let us introduce an important class of stable solutions and the known results on their regularity.

%%%%%%%%%%%%%%%%%%%%%%%%%%%%%%%%%%%%%%%%%%%%%%%%%%%%%%%%%%%%%%%%%%%%%%%
\subsection{Extremal solutions. Known results for $\mathbf{\s\in (0,1)}$}
\label{Subsec:Extremal}
%%%%%%%%%%%%%%%%%%%%%%%%%%%%%%%%%%%%%%%%%%%%%%%%%%%%%%%%%%%%%%%%%%%%%%%

Given a power $\s \in (0,1)$, consider the  problem
\begin{equation}
	\label{Eq:ExtremalProblem}
	\beqc{\PDEsystem}
	\fraclaplacian u &= &\lambda f(u) &  \text{in }  \Omega ,\\
	u &= &0 &  \text{in }    \R^n\setminus\Omega,
	\eeqc
\end{equation}
where $\Omega \subset \R^n$ is a bounded smooth domain, $\lambda > 0$ is a real parameter, and 
\begin{equation}
	\label{Eq:Conditionsf}
	f\in C^1([0,\infty)), \ f\text{ is nondecreasing, } f(0)>0, \text{ and } \lim_{t \rightarrow +\infty} \dfrac{f(t)}{t} = + \infty .
\end{equation}
It is known (see \cite{RosOtonSerra-Extremal}) that problem \eqref{Eq:ExtremalProblem} admits an increasing family of minimal\footnote{Here, \emph{minimal} means that $u_\lambda$ is smaller than any other solution (and, as a byproduct, smaller than any supersolution).} stable solutions $\setcond{u_\lambda}{0<\lambda<\lambda^\star}$, with $u_\lambda > 0$ being bounded in $\Omega$, up to a certain finite extremal parameter~$\lambda^\star$. For $\lambda > \lambda^\star$ there is no solution to \eqref{Eq:ExtremalProblem}, even in the $L^1$-weak sense.\footnote{\label{Footnote:L1weak}We say that $ u \in L^1(\Omega)$ is a stable $L^1$-weak solution of \eqref{Eq:ExtremalProblem} if $f(u) \dist (\cdot, \partial \Omega)^s \in L^1(\Omega)$ and
	\begin{equation}
		\label{Eq:DefinitionL1weakSol}
		\int_\Omega u \, \fraclaplacian \zeta \d x' = \int_\Omega \lambda f(u) \zeta \d x' 
	\end{equation}
	for all $\zeta$ such that $\zeta$ and $\fraclaplacian \zeta$ are bounded in $\Omega$ and $\zeta \equiv 0$ in $\R^n\setminus \Omega$.}
The pointwise limit of $\{u_\lambda\}$, as $\lambda \nearrow \lambda^\star$, is a stable $L^1$-weak solution of \eqref{Eq:ExtremalProblem} for $\lambda=\lambda^\star$. 
Such solution, denoted by $u^\star$, is called the \emph{extremal solution} of \eqref{Eq:ExtremalProblem}. 
A fundamental question is to determine whether it is bounded or not.

The main result of our  paper can be used (through \Cref{Coro:LinftyConvexDiriclet}) to establish the boundedness of the extremal solution in dimensions $n\leq 4$ for $\s=1/2$ when $\Omega$ is a convex domain and $f$ is a convex nonlinearity.

\begin{corollary}
	\label{Coro:Extremal}
	For some $\gamma >0$, let $f$ be a $C^{1,\gamma}$ convex function  satisfying \eqref{Eq:Conditionsf}.
	Let $\Omega\subset \R^n$ be a bounded convex $C^1$ domain and let $u^\star$ be the extremal solution to \eqref{Eq:ExtremalProblem} for $\s = 1/2$.
	
	If $1 \leq n \leq 4$, then $u^\star \in L^\infty(\Omega)$.
\end{corollary}

The result is established by applying the estimate of \Cref{Coro:LinftyConvexDiriclet}, for $\lambda < \lambda^\star$, to the bounded solutions $u_\lambda$. 
Then, since it is easy to prove that $u^\star\in L^1(\Omega)$, the estimates for $u_\lambda$ are uniform in $\lambda$ and one can take the limit $\lambda \to \lambda^\star $ to deduce the boundedness of~$u^\star$.

Besides \Cref{Coro:Extremal}, the available results on the boundedness of the extremal solution for the fractional problem \eqref{Eq:ExtremalProblem} are those contained in \cite{RosOtonSerra-Extremal}, \cite{RosOton-Gelfand}, and \cite{SanzPerela-Radial}.
We briefly describe them next.
See \Cref{Fig:KnownResults} for a schematic representation of the dimensions achieved in each result.

In~\cite{RosOtonSerra-Extremal}, Ros-Oton and Serra showed that if $f$ is convex then $u^\star$ is bounded whenever $n < 4s$, extending to the nonlocal setting the arguments of Nedev~\cite{Nedev}.
The authors also established the boundedness of the extremal solution in dimensions $n < 10s$ under the more restrictive assumption of $f$ being $C^2$ and $f f'' /(f')^2$ having a limit at infinity; this is an extension of the arguments of Crandall and Rabinowitz~\cite{CrandallRabinowitz}. 
The estimates for stable solutions in~\cite{RosOtonSerra-Extremal} depend on $f$, in contrast with ours, and do not provide any estimate for $n=1$ if $\s$ is small. 
The results from \cite{RosOtonSerra-Extremal} have been recently extended to the case of systems of two equations by  Fazly~\cite{Fazly-Extremal}.

Shortly after \cite{RosOtonSerra-Extremal}, Ros-Oton~\cite{RosOton-Gelfand} obtained an optimal result  in the case of the exponential nonlinearity $f(u) = e^u$, extending the arguments of~\cite{DavilaDupaigneMontenegro} commented below, by showing that $u^\star$ is bounded whenever\footnote{Condition \eqref{Eq:OptimalDimensionsGammas} is equivalent to 
	$$
	|x'|^\frac{2\s + n }{2}  \fraclaplacian |x'|^\frac{2\s - n }{2}  < \lim_{\varepsilon \to 0} \dfrac{2 \s}{\varepsilon} |x'|^{n - \varepsilon}\fraclaplacian |x'|^{2\s - n + \varepsilon} \quad \text{ for } x'\in\R^n \setminus \{ 0\},
	$$
	which is indeed the expression arising in \cite{RosOton-Gelfand}.
	Note that the two quantities in this inequality are constants independent of $|x'|$. 
}
\begin{equation}
	\label{Eq:OptimalDimensionsGammas}
	\dfrac{\Gamma^2(\frac{n + 2s}{4})}{\Gamma^2(\frac{n - 2s}{4})} < \dfrac{\Gamma(\frac{n}{2})\Gamma(1+s)}{\Gamma(\frac{n - 2s}{2})}
\end{equation}
and $\Omega$ is symmetric and convex with respect to all the coordinate directions.
In \Cref{Fig:KnownResults} below we represent graphically the dimensions given by \eqref{Eq:OptimalDimensionsGammas}.
In particular, we see that \eqref{Eq:OptimalDimensionsGammas} holds in dimensions $n \leq 7$ for all $s \in (0,1)$, and additionally in dimension $n=8$ for some fractions $\s$ which include $\s=1/2$.

\begin{figure}[h]
	\input{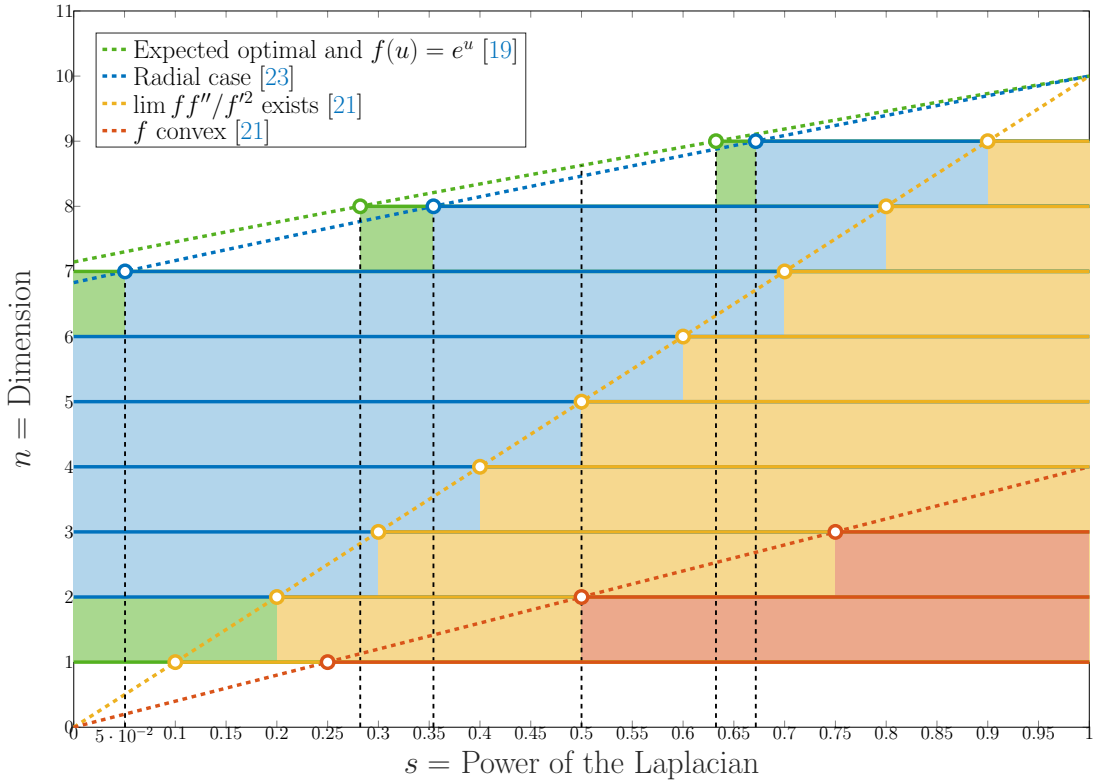}	
	\caption{Graphical representation of the known results on the boundedness of the extremal solution.
		For each color, the horizontal solid lines represent the rank of $\s$ for which   $u^\star$ is bounded in that dimension.
		The dashed colored lines represent the threshold values $(n,\s)$ found in each paper: in green, the values given by inequality \eqref{Eq:OptimalDimensionsGammas}; in blue, $n= 2(\s+2+\sqrt{2(\s+1)})$; in yellow, $n= 10\s$; in red, $n= 4\s$.
		These lines correspond, in descending order, to the four cases listed in the legend.
		The black dashed vertical lines help to locate some limiting values of $\s$.}\label{Fig:KnownResults}
\end{figure}

Condition \eqref{Eq:OptimalDimensionsGammas} is optimal since, when \eqref{Eq:OptimalDimensionsGammas} fails, $\log |x'|^{-2s}$ is a singular stable solution to $\fraclaplacian u = \lambda e^u$ in all of $\R^n$ for some $\lambda>0$ (see Remark~2.2 of~\cite{RosOton-Gelfand}).

One expects \eqref{Eq:OptimalDimensionsGammas} to provide the optimal range for the values of $n$ and $\s$ in which stable solutions are bounded also for all nonlinearities $f$ (or for a large class of them).
Nevertheless, this is not known even in the radial case,\footnote{For radially symmetric solutions, the local equation $-\Delta u= f(u)$ becomes a second order ODE. 
	Instead, although $\fraclaplacian u= f(u)$ can be written as a one-dimensional equation in the radial case, it still involves an integro-differential operator which is much more involved than an ODE.} as explained next.

The problem in a ball is studied in a recent paper by Sanz-Perela~\cite{SanzPerela-Radial}.
It is proved that in dimensions $2\leq n< 2(\s+2+\sqrt{2(\s+1)})$ the extremal solution to \eqref{Eq:ExtremalProblem}, with $\Omega= B_1$, is bounded.
This extends the results of Cabré and Capella~\cite{CabreCapella-Radial} for the local case, and of Capella, Dávila, Dupaigne, and Sire~\cite{CapellaEtAl} for another nonlocal operator. 
The condition on the dimension in this radial case (in blue in \Cref{Fig:KnownResults}) is slightly worse than the expected optimal one \eqref{Eq:OptimalDimensionsGammas} (in green), but it is a significant improvement of $n<10\s$ (in yellow).
In \Cref{Subsec:IdeasOfProofs}  we will explain why \cite{SanzPerela-Radial} and the current work do not reach the expected optimal dimensions.

A further question of interest is the partial regularity of singular stable solutions.
In another very recent work, Hyder and Yang~\cite{HyderYang-SingularSet} have established that the dimension of the singular set of a stable solution to $\fraclaplacian u = e^u$ in $\Omega$ is at most $n-10\s$.

Condition \eqref{Eq:OptimalDimensionsGammas} in the case $\s=1/2$ appeared in the literature for the first time in the paper~\cite{DavilaDupaigneMontenegro} by Dávila, Dupaigne, and Montenegro.
They dealt with a boundary reaction problem which is related to, but different from, our equation $\halflaplacian u = \partial_v v = f(v) $ on $\partial \R^{n+1}_+$ (where $v$ is the harmonic extension of $u$ in the half-space).
Besides the case $f(u) = \lambda e^u$, \cite{DavilaDupaigneMontenegro} also studied power-like nonlinearities, for which a more precise condition on the dimension, now depending on the power $p$, was found.
It corresponds to \eqref{Eq:GammasP} below for $\s= 1/2$.

\subsection{Entire solutions}
\label{Subsec:EntireSolutions}

A related issue to our problem is the classification of entire solutions to $\fraclaplacian u = f(u)$ in all of $\R^n$ which are stable or, more generally, have finite Morse index.
This problem has been treated very recently in a series of works, in which the condition \eqref{Eq:OptimalDimensionsGammas} appears again.

For the power case $f(u) = |u|^{p-1}u$, Dávila, Dupaigne, and Wei~\cite{DavilaDupaigneWei-LaneEmden} give a complete classification of finite Morse index solutions.
Namely, they prove that when either $1 < p < (n + 2\s)/(n - 2\s)$, or $p > (n + 2\s)/(n - 2\s)$ and
\begin{equation}
	\label{Eq:GammasP}
	\dfrac{\Gamma^2(\frac{n + 2s}{4})}{\Gamma^2(\frac{n - 2s}{4})} < p \dfrac{\Gamma ( \frac{n}{2} - \frac{\s}{p-1}) \Gamma ( \s +  \frac{\s}{p-1}) }{  \Gamma ( \frac{\s}{p-1}) \Gamma ( \frac{n-2\s}{2} - \frac{\s}{p-1}) },
\end{equation}
every finite Morse index entire solution is identically zero.
%A similar result holds for $f(x,u) = |x|^a |u|^{p-1}u$ with a condition on $n$, $\s$, $a$, and $p$ which generalizes \eqref{Eq:GammasP}, see~\cite{FazlyWei-HenonLaneEmden}.
Note that taking the limit $p\to +\infty$ in \eqref{Eq:GammasP}, one obtains precisely the expected optimal condition \eqref{Eq:OptimalDimensionsGammas}.

In the case $f(u) = e^u$, Hyder and Yang~\cite{HyderYang-ClassGelfand} have recently proved that no finite Morse index entire solution exists if \eqref{Eq:OptimalDimensionsGammas} holds ---Duong and Nguyen~\cite{DuongNguyen-liouville} had previously reached the condition $n < 10\s$.
%A similar result holds for $f(x,u) = |x|^a e^u$ with a condition on $n$, $\s$, and $a$ which generalizes \eqref{Eq:OptimalDimensionsGammas}, see~\cite{FazlyYeyaoYang-HenonGelfand}.

Our main result allows to establish a Liouville theorem for stable solutions in dimensions $n\leq4$ for general nonnegative convex nonlinearities. We follow the ideas of Dupaigne and Farina~\cite{DupaigneFarina} for the local case, which consist of applying a universal Hölder estimate to the blow-downs of a stable solution.
\begin{corollary}
	\label{Coro:Liouville}
	Let $u\in C^2(\R^n)\cap L^1_{1/2}(\R^n)$ be a stable solution to $(-\Delta)^{1/2}u = f(u)$ in $\R^n$, where $f$ is a nonnegative convex $C^{1,\gamma}$ function for some $\gamma >0$.
	Assume that 
	\begin{equation}
		\label{Eq:LowerBoundu}
		u(x') \geq - c_0  \log (2 + |x'|) \quad \text{ for } x'\in \R^n,
	\end{equation}
	where $c_0 >0$ is a constant.
	
	If $1\leq n\leq 4$, then $u$ must be constant.
	
\end{corollary}

\subsection{Outline of the proofs}
\label{Subsec:IdeasOfProofs}

Here we summarize the main ideas in the proofs of \Cref{Th:Holder,Th:GeomStabilityFullGradient_s=1/2}.
We use the extension problem for the half-Laplacian as a main tool. 
For this, let us settle our notation.

Through all the paper we will denote points in $\R^{n+1}_+ = \R^n \times (0,+\infty)$ by $x= (x',y)$.
Recall that the half-Laplacian of a regular enough function $u:\R^n \to \R$ can be computed through its harmonic extension in $\R^{n+1}_+$. 
Indeed, if $v$ solves $\Delta v = 0$ in $\R^{n+1}_+$ and $v(\cdot,0) = u$, then $- v_y(\cdot,0) = \halflaplacian u$.

We will denote  balls in $\R^n$ and $\R^{n+1}$ by $B_R$ and $\B_R$ respectively, while 
$$
\B_{R}^+ := \B_R \cap \R^{n+1}_+ = \setcond{x=(x',y)\in \R^{n+1}_+}{|(x',y)|< R}
$$
will be half-balls in $\R^{n+1}_+$.
If $x_0'\in \R^n$, $\B_{R}^+(x_0'):=(x_0',0)+\B_{R}^+$.
In addition, we will use the notation
$$
r := |x| \quad \text{ and } \quad
v_r := \dfrac{x}{r}\cdot \nabla v (x) ,  \quad \text{ for } x \in \R^{n+1}_+.
$$

The general strategy to obtain estimates for stable solutions consists of choosing an appropriate test function in the stability condition \eqref{Eq:Stability}, which for $\s = 1/2$ can be written, using the extension problem, as
\begin{equation}
	\label{Eq:StabilityExtensionHalfLaplacian}
	\int_\Omega f'(u) \xi^2 \dx' \leq \int_{\R_+^{n+1}}  |\nabla \xi |^2 \d x
\end{equation}
for every $\xi \in H^1( \R^{n+1}_+)$ such that its trace in $\R^n$ has compact support in $\Omega$; see \Cref{Sec:Step1ForAlls}.
The choices of test function  taken in each of the previously mentioned papers \cite{RosOtonSerra-Extremal,RosOton-Gelfand,SanzPerela-Radial} were the following (up to a cut-off): $\xi = h(u)$ for some $h$ depending on $f$ in \cite{RosOtonSerra-Extremal}, $\xi = |x'|^{- a}$ for some $a>0$ in \cite{RosOton-Gelfand}, and $\xi = |x'|^{-b} (x' \cdot \nabla_{x'} v)$ for some $b>0$ in \cite{SanzPerela-Radial}. 
Instead, we will make use of two new test functions, \eqref{Eq:TestFunctionRadialDerivative} and \eqref{Eq:TestFunctionFullGrad} below.

Our proof of \Cref{Th:Holder} follows the main lines of that in \cite{CabreFigalliRosSerra-Dim9} for the local case, but confronts a delicate issue in a compactness argument (which remains open for $\s\neq 1/2$).
The main arguments in the proof can be summarized in four steps:

%%%%%%%%%%%%%%%%%%%%%%%%%%%%%%%
$ \bullet$  \textbf{Step 1.}
%%%%%%%%%%%%%%%%%%%%%%%%%%%%%%%
A key point is to choose 
\begin{equation}
	\label{Eq:TestFunctionRadialDerivative}
	\xi = r^{1 + (1 - n)/2} v_r \zeta = |x|^{-(n-1)/2} (x\cdot \nabla v) \zeta 
\end{equation}
as test function in the stability condition \eqref{Eq:StabilityExtensionHalfLaplacian}, where $v$ is the harmonic extension of $u$ and $\zeta$ is a smooth cut-off function. 
This leads to the following bound (see \Cref{Prop:EstimatevrAnnulus}):
\begin{equation}
	\label{Eq:EstimatevrAnnulusHalfLap}
	\int_{\B_{1/2}^+}  r^{1 - n} v_r^2 \d x \leq C \int_{\B_{3/4}^+ \setminus \B_{1/2}^+}  |\nabla v|^2  \d x \quad \text{ if }  2 \leq n \leq 4.
\end{equation}
The crucial fact here is that, after this choice of $\xi$, the nonlinearity $f$ no longer appears in the estimates.

The condition $2 \leq n \leq 4$ (which for general $\s$ reads $2\s < n < 10 \s$ in \Cref{Prop:EstimatevrAnnulus} below) is not the expected optimal range of dimensions.
The reason why our proof does not reach such range is that, although we believe that the choice $\xi |_{\{y= 0\}} = |x'|^{-(n-1)/2}(x' \cdot \nabla_{x'} u)\zeta$ should be optimal, the election of its extension \eqref{Eq:TestFunctionRadialDerivative} is not.
The same issue occurs in the radial case treated in \cite{SanzPerela-Radial}, where still another different extension of the test function $\xi |_{\{y= 0\}}$ is chosen.

%%%%%%%%%%%%%%%%%%%%%%%%%%%%%%%
$ \bullet$  \textbf{Step 2.}
%%%%%%%%%%%%%%%%%%%%%%%%%%%%%%%
To relate the left and right-hand sides of \eqref{Eq:EstimatevrAnnulusHalfLap},  we prove that, under a doubling condition on $|\nabla v|^2\d x$, the quantities $v_r$ and $\nabla v $ are comparable, in $L^2$, in an annulus (see \Cref{Lemma:v_rControlsGradientInAnnulus}).
This combined with \eqref{Eq:EstimatevrAnnulusHalfLap} will lead to the geometric-decay bound
\begin{equation}
	\label{Eq:RadialDecayvrIntro}
	\int_{\B_R^+}  r^{1 - n} v_r^2 \d x 
	%= \int_{\B_R^+}  |x|^{- (n+1)} (x\cdot \nabla v)^2 \d x 
	\leq C  R^{2 \alpha} \norm{\nabla v}^2_{L^2(\B^+_{3/4})} \quad \text{ if } 2 \leq n \leq 4,
\end{equation}
for all small enough radii $R$ and for some $\alpha > 0$.

In order to establish the comparability of $v_r$ and $\nabla v$ in annuli we proceed as in~\cite{CabreFigalliRosSerra-Dim9}, using a compactness argument.
To carry it on, a control on higher order Sobolev norms of $v$ is needed, and for this it is crucial to use the new  stability condition in geometric form, \eqref{Eq:GeomStabilityFullGradient_s=1/2}, given in \Cref{Th:GeomStabilityFullGradient_s=1/2}.

To establish \Cref{Th:GeomStabilityFullGradient_s=1/2} we choose
\begin{equation}
	\label{Eq:TestFunctionFullGrad}
	\xi=|\nabla v|\eta 
\end{equation}
in the stability condition.
At some point along the proof, to show the sign of a certain term, we use crucially the convexity and nonnegativeness of $f$.

%%%%%%%%%%%%%%%%%%%%%%%%%%%%%%%
$ \bullet$  \textbf{Step 3.}
%%%%%%%%%%%%%%%%%%%%%%%%%%%%%%%
Combining the decay of the weighted radial derivative and a Morrey-type estimate from \cite{CabreQuantitative} (\Cref{Lemma:MorreyAverage}), we obtain the Hölder bound
\begin{equation}
	\label{Eq:HolderEstimateExt}
	\seminorm{u}_{C^\alpha( \overline{B}_{1/2})} \leq C
	\norm{\nabla v}_{L^2(\B^+_{3/4})} \quad \text{ if } 2 \leq n \leq 4.
\end{equation}

At this point, it will only remain to control the $L^2$ norm of $\nabla v$ by the $L^1_{1/2}(\R^n)$ norm of~$u$. This is done in \Cref{Sec:H1Control}.
For this, the key point is to realize that, thanks to \Cref{Th:GeomStabilityFullGradient_s=1/2}, we have a control of the $L^2$ norm of $D^2v$ in terms of a lower order norm of~$v$.
Using this and two interpolation results from~\cite{CabreQuantitative} (proved  in \Cref{Sec:Interpolation} below), we control the $H^1$ norm of~$v$ by its $L^1$ norm and, in turn, by the $L^1_{1/2}$ norm of $u$. 

This will also lead,  through a trace inequality, to the $H^{1/2}$ bound \eqref{Eq:H^1/2Estimate} for $u$ in all dimensions.

%%%%%%%%%%%%%%%%%%%%%%%%%%%%%%%
$ \bullet$  \textbf{Step 4.}
%%%%%%%%%%%%%%%%%%%%%%%%%%%%%%%
Finally, to obtain the result in dimension $n=1$, we just need to consider the solution as defined in $\R^2$ after the addition of an artificial variable, and apply the result for $n=2$.
In \Cref{Sec:AddDimensions} we establish the necessary results to carry out this process.

%%%%%%%%%%%%%%%%%%%%%%%%%%%%%%%%%%%%%%%%
\subsection{Plan of the article}
%%%%%%%%%%%%%%%%%%%%%%%%%%%%%%%%%%%%%%%%

In \Cref{Sec:Step1ForAlls} we establish the key estimate \eqref{Eq:EstimatevrAnnulusHalfLap} of Step~1, which we carry on in the general setting of the equation $\fraclaplacian u = f(u)$ with $\s\in (0,1)$.
In \Cref{Sec:GeomStability} we establish the new geometric stability inequality given by \Cref{Th:GeomStabilityFullGradient_s=1/2}, while 
\Cref{Sec:GeometricDecayGradient} is devoted to Step~2 described above.
Then, in \Cref{Sec:H1Control} we establish the $H^{1/2}$ estimate of Step~3, leading to the proofs of \Cref{Th:Holder} and \Cref{Coro:LinftyConvexDiriclet,Coro:Extremal,Coro:Liouville} in \Cref{Sec:LinftyHalfLaplacian}.
Finally, in \Cref{Sec:AddDimensions} we present some results related to adding an artificial variable.

In the appendices we collect some results that are used through the paper, and we recall the proofs of the two interpolation inequalities from \cite{CabreQuantitative} needed in \Cref{Sec:H1Control}.

%%%%%%%%%%%%%%%%%%%%%%%%%%%%%%%%%%%%%%%%%%%%%%%%%%%%%%%%%%%%%%%%%%%%%%%%%%%%
%%%%%%%%%%%%%%%%%%%%%%%%%%%%%%%%%%%%%%%%%%%%%%%%%%%%%%%%%%%%%%%%%%%%%%%%%%%%
\section{A key weighted radial derivative estimate}
%%%%%%%%%%%%%%%%%%%%%%%%%%%%%%%%%%%%%%%%%%%%%%%%%%%%%%%%%%%%%%%%%%%%%%%%%%%%
%%%%%%%%%%%%%%%%%%%%%%%%%%%%%%%%%%%%%%%%%%%%%%%%%%%%%%%%%%%%%%%%%%%%%%%%%%%%
\label{Sec:Step1ForAlls}

The goal of this section is to prove estimate~\eqref{Eq:EstimatevrAnnulusHalfLap}.
Although in this paper we only consider the case $\s=1/2$, for future reference  all the computations in this section are done for the equation $\fraclaplacian u = f(u)$ in $B_1$, with $\s\in (0,1)$.

Along this section we will assume that $u\in C^2(B_1)\cap L^1_\s(\R^n)$ and that it solves (pointwise) the equation $\fraclaplacian u = f(u)$ in $B_1$, with $f\in C^{1,\gamma}$ for some $\gamma >0$. 
In the following remark we comment on these assumptions and discuss some extensions of our results to less regular classes of stable solutions.

\begin{remark}
	\label{Remark:Regularity}
	We can consider three classes of stable solutions to \eqref{Eq:SemilinearEquation} in terms of their regularity: $L^1$-weak solutions ---defined in footnote~\ref{Footnote:L1weak}---, energy solutions ---critical points of the functional $E(\cdot)$ in \eqref{Eq:Energy}---, and pointwise (or classical) solutions.
	As mentioned in Remark~2.1 of \cite{RosOtonSerra-Extremal} (see \cite{SanzPerela-Approx} for nonzero values in $\R^n \setminus \Omega$), these three notions of solution coincide whenever the solution is bounded and $f$ sufficiently regular.\footnote{Indeed, if $u$ is an $L^1$-weak solution to $\fraclaplacian u = f(u)$ in $B_1$ which is bounded in $B_1$ (note that in the paper \cite{RosOtonSerra-Extremal} $L^1$-weak solutions are called \textit{weak solutions}), and $f\in C^{1,\gamma}$ for some $\gamma >\max\{0, 1-2\s\}$, then by considering the convolution of $u$ with a standard mollifier and using regularity results for the fractional Laplacian it follows readily that $u\in C^2(B_1)$ ---see Corollaries~2.3 and 2.5 in \cite{RosOtonSerra-Regularity}; note that here we do not use the stability of $u$. }
	Thus, our main result holds for bounded $L^1$-weak or energy solutions. 
	
	When $u$ is not assumed a priori to be bounded, the situation is more delicate.
	On the one hand, our result leads to the interior Hölder regularity of energy stable solutions or, more generally, $L^1$-weak stable solutions which belong to the energy space associated to $E(\cdot)$.
	Indeed, it can be proved that any stable solution $u$ in these classes can be approximated by bounded stable solutions $u_k$ (at least when $f$ is increasing and convex and $\s \geq 1/2$; see~\cite{SanzPerela-Approx}).
	Then, the estimates for $u_k$ pass on the limiting function $u$.
	
	On the other hand, as in the local case $\s=1$, our main result does not hold for general $L^1$-weak solutions, as there exist unbounded $L^1$-weak solutions which satisfy the stability hypothesis in the dimensions given by \Cref{Th:Holder}; see~\cite{SanzPerela-Approx}. 
	The issue here is that these unbounded $L^1$-weak solutions are not in the energy space associated to $E(\cdot)$.
	
\end{remark}

To state the main result of this section, we need to introduce the extension problem for all the fractional powers $\s\in(0,1)$ of the Laplacian. 
Recall that if $v$ solves
\begin{equation}
	\label{Eq:ExtensionProblemDirichlet}
	\beqc{rcll}
	\div (y^a \nabla v ) &= & 0 & \text{in } \R^{n+1}_+ ,\\
	v\, &= &u &  \text{on } \partial \R^{n+1}_+ = \R^n,
	\eeqc
\end{equation}
where $a:= 1 - 2\s$, then
\begin{equation}
	\label{Eq:DirichletToNeumannRelation}
	\dfrac{\partial v}{\partial \nu^a} := - \lim_{y \downarrow 0} y^a v_y = \dfrac{1}{d_\s} \fraclaplacian u 
\end{equation}
for a positive constant $d_\s$ which depends only on $\s$, and such that $d_{1/2}=1$. 
The function~$v$, which is smooth in $\R^{n+1}_+ $, is called the $\s$-\emph{harmonic extension} of $u$ and can be expressed in terms of $u$ through convolution with a Poisson kernel ---see \eqref{Eq:vPoisson} in our first appendix.
Along the paper we will always denote by $v$ the $\s$-harmonic extension of $u$.

\begin{proposition}
	\label{Prop:EstimatevrAnnulus}
	Let $n\geq 1$,  $\s \in (0,1)$, and let $u\in C^2(B_1)\cap L^1_\s(\R^n)$ be a stable solution to $\fraclaplacian u = f(u)$ in $B_1\subset \R^n$, where $f$ is a  $C^{1,\gamma}$ function for some $\gamma >0$.
	Let $v$ be the $\s$-harmonic extension of $u$. 
	
	Then, 
	\begin{equation}
		\label{Eq:EstimatevrAnnulus}
		\int_{\B_{1/2}^+} y^{1-2\s} r^{2\s - n} v_r^2 \d x \leq C \int_{\B_{3/4}^+ \setminus \B_{1/2}^+} y^{1-2\s} |\nabla v|^2  \d x \quad \quad \text{if } n \in (2\s, 10\s),
	\end{equation}
	for some constant $C$ depending only on $n$ and $\s$.
\end{proposition}

This result will follow from making a particular choice of the test function in the stability condition, once this condition is written in the extended space $\R_+^{n+1}$.
Namely, definition \eqref{Eq:Stability} is equivalent to requiring
\begin{equation}
	\label{Eq:StabilityExtension}
	\int_\Omega f'(u) \xi^2 \dx' \leq d_\s\int_{\R_+^{n+1}} y^a |\nabla \xi |^2 \d x
\end{equation}
for every $\xi \in H^1( \R^{n+1}_+, y^a)$ whose trace in $\R^n$ has compact support in $\Omega$. 
This follows easily from the fact that the constant $d_s$ in \eqref{Eq:StabilityExtension}, which also appears in the previous relation~\eqref{Eq:DirichletToNeumannRelation}, is the optimal constant in the trace inequality $\seminorm{\xi(\cdot, 0)}_{H^\s (\R^n)}^2 \leq d_\s \seminorm{\xi}_{H^1 (\R^{n+1}_+, y^a)}^2$; see \cite[Section~5]{FranKLenzmannSilvestre}.

The following lemma is the first step towards \Cref{Prop:EstimatevrAnnulus}.
It is proved by taking $\xi = (x \cdot \nabla v )\eta = r v_r \eta$ in the stability condition \eqref{Eq:StabilityExtension}. 
The crucial point here is that, after this particular choice, the nonlinearity $f$ no longer appears in the inequality.

\begin{lemma}
	\label{Lemma:Stabilityc=rv_r}
	Let $n\geq 1$, $\s \in (0,1)$, and let $u\in C^2(B_1)\cap L^1_\s(\R^n)$ be a stable solution to $\fraclaplacian u = f(u)$ in $B_1\subset \R^n$, where $f$ is a  $C^{1,\gamma}$ function for some $\gamma >0$.
	Let $v$ be the $\s$-harmonic extension of $u$. 
	
	Then, for every $\eta \in \Lip(\overline{\R^{n+1}_+})$ with compact support in $\B_1^+ \cup B_1$ it holds
	\begin{equation}
		\label{Eq:Stabilityc=rv_r}
		\begin{split}
			\s \int_{\B_1^+} y^a & \left \{ (n - 2\s)\eta^2 + r(\eta^2)_{r} \right \} |\nabla v|^2 \dx 
			- 2\s  \int_{\B_1^+} y^a r v_{r} \nabla v \cdot \nabla (\eta^2) \d x \\
			&\quad \leq \int_{\B_1^+} y^a r^2 v_{r}^2 |\nabla \eta|^2 \d x .
		\end{split}
	\end{equation}
\end{lemma}

\begin{proof}
	Throughout the proof we will use the notation $v_i = \partial_i v = \partial_{x_i} v$, $v_y = v_{n+1}$, and $x_{n+1} = y$.

	The key idea is to take  $\xi = \textbf{c} \eta$ in the stability condition \eqref{Eq:StabilityExtension}, with $\textbf{c} = x \cdot \nabla v= r v_r $, and use that $\textbf{c}$ satisfies an appropriate equation for the linearized operator $d_\s \partial_{\nu^a} - f'(u)$ (see \eqref{Eq:EquationsForC} below).
	To carry out this program we need to prove some regularity for $\textbf{c}$.
	In what follows, we will use that, for every $R<1$, the functions $v$, $y^a \partial_y v$, $\nabla_{x'}v$, and $y^a \partial_y \nabla_{x'} v$ are continuous in $\overline{\B_R^+}$, that  $D^2_{x'}v$ is bounded in $\B_R^+$, and that $v\in H^1(\B_R^+, y^a)$.
	All these statements are proved in \Cref{Lemma:RegularityHorizontalGrad}.

	Regarding $\textbf{c}$, we next show two properties of this function.
	First, note that $\textbf{c}$ is continuous up to $\{y=0\}$ and $\textbf{c}(\cdot, 0) = x'\cdot \nabla_{x'}u$ in $B_R \subset \{y=0\}$ for $R<1$.
	Indeed,  $\textbf{c} = x'\cdot \nabla_{x'} v + y v_y$ and, since $\nabla_{x'}v$ and $y^a v_y$ are continuous up to $\{y = 0\}$ in $\B^+_R$,  the claim follows using that $|y v_y| = |y^a v_y y^{1-a}| \leq C y^{2\s} \to 0$ as $y\downarrow 0$.

	Second, we show that $\textbf{c} \in H^1 (\B^+_R, y^a)$ for all $R<1$, and as a consequence $\xi = \textbf{c}  \eta\in H^1 (\B^+_R, y^a)$ and thus it is an admissible test function for the stability condition \eqref{Eq:StabilityExtension} ---since $\eta$ is $C^1$ and has compact support in $\B^+_1 \cup B_1$.
	To see that $\textbf{c} \in H^1 (\B^+_R, y^a)$, recall that $v \in H^1 (\B^+_R, y^a)$ for $R<1$, and thus $\textbf{c} \in L^2 (\B^+_R, y^a)$.
	Hence, we only need to check that $\nabla \textbf{c} \in L^2 (\B^+_R, y^a)$.
	For $i= 1,\ldots,n$, we have
	\begin{equation}
		\textbf{c}_{x_i} = v_{x_i} + x'\cdot \nabla_{x'} v_{x_i} + y v_{x_i y}.
	\end{equation} 
	Clearly the first two terms belong to $L^2 (\B^+_R, y^a)$ since $\nabla_{x'}v$ and $D^2_{x'} v$ are bounded in $\B^+_R$  and $y^a$ is integrable near $0$.
	For the last term we use that $y^a v_{x_i y}$ is continuous up to $\{y=0\}$ to obtain that $|v_{x_i y}| \leq C y^{-a}$ for some constant $C$, and thus $ y^a |y v_{x_i y}|^2 \leq C y^{2-a} = C y^{1 + 2\s}$, which yields the desired integrability.
	Finally, for the vertical derivative of $\textbf{c}$ we use the equation $\Delta v + a v_y / y = 0$ to obtain
	\begin{equation}
		\label{Eq:FluxcFory>0}
		\textbf{c}_{y} =  x'\cdot \nabla_{x'} v_{y}  + v_{y} + y v_{y y} =   x'\cdot \nabla_{x'} v_{y} + (1- a) v_{y} - y \Delta_{x'} v,
	\end{equation} 
	and using similar arguments as the previous ones, we see that all the terms belong to $L^2 (\B^+_R, y^a)$.
	
	Note also that, since $y^a v_y$, $y^a \nabla_{x'} v_{y}$, and $y^{2 - 2\s }\Delta_{x'} v$ are continuous up to $\{y=0\}$ (in $\B^+_R$), \eqref{Eq:FluxcFory>0} leads to the flux $y^a \textbf{c}_y$ being continuous and well defined up to $B_R \subset \{ y = 0\}$.

	We can now take $\xi = \textbf{c}  \eta$ in the stability condition \eqref{Eq:StabilityExtension} to get
	\begin{equation}
		\label{Eq:StabilityConditionCEtaWeak}
		\int_{B_1} f'(u) \textbf{c}^2  \eta^2 \dx' \leq d_\s\int_{\B^+_1} y^a \nabla \textbf{c} \cdot \nabla (\textbf{c} \eta^2) \d x + d_\s\int_{\B^+_1} y^a \textbf{c}^2  |\nabla \eta |^2 \d x.
	\end{equation}
	We claim that
	\begin{equation}
		\label{Eq:IntegrationByPartsC}
		d_\s \int_{\B^+_1} y^a \nabla \textbf{c} \cdot \nabla (\textbf{c} \eta^2) \d x = \int_{B_1} \left( f'(u)\textbf{c} + 2\s f(u) \right) \textbf{c} \eta^2 \d x'.
	\end{equation}
	To see this, we first show that
	\begin{equation}
		\label{Eq:EquationsForC}
		\div(y^a \nabla \textbf{c}) = 0 \,  \textrm{ in } \R_+^{n+1} \quad \textrm{ and } \quad d_\s \dfrac{\partial \textbf{c}}{\partial \nu^a} - f'(u)\textbf{c} =   2\s f(u) \, \textrm{ in } B_1 .
	\end{equation}
	The first identity is obtained after a simple computation using that  $\Delta v + a v_y / y = 0$:
	\begin{align}
		y^{-a} \div(y^a \nabla \textbf{c})
		&= \Delta (x \cdot \nabla v) + \dfrac{a}{y}( x\cdot \nabla v)_y\\
		&= x \cdot \nabla \left(-a \dfrac{v_y}{y}\right) + 2 \Delta v + \dfrac{a}{y} (v_y + x\cdot \nabla v_y)\\
		&= -a  x \cdot \nabla \left( \dfrac{v_y}{y}\right) +  \Delta v + \dfrac{a}{y}  x\cdot \nabla v_y = a y \dfrac{v_y}{y^2} + \Delta v = 0.
	\end{align}
	To show the second one, recall that from \eqref{Eq:FluxcFory>0}, for $y>0$ we have
	\begin{equation}
		-y^a \textbf{c}_y = x' \cdot (- y^a \partial_y \nabla_{x'}v) + 2\s (-y^a v_y) + y^{2 - 2\s} \Delta_{x'} v.
	\end{equation}
	Now, the desired identity follows after taking the limit $y\downarrow 0$ in the above expression, by using \Cref{Lemma:RegularityHorizontalGrad}~($b$), that $d_\s \partial_{\nu^a} v = f(u)$ in $B_1$, and that $\Delta_{x'} v$ is bounded ---and thus $|y^{2 - 2\s} \Delta_{x'} v|\leq C y^{2 - 2\s} \to 0$ as $y\downarrow 0$.
	We have established \eqref{Eq:EquationsForC}.
	
	To prove \eqref{Eq:IntegrationByPartsC}, we integrate by parts in $\B^+_R \cap \{ y > \delta \}$ (a set where $\textbf{c}$ is smooth) and use the first identity in \eqref{Eq:EquationsForC} to obtain
	\begin{equation}
		\label{Eq:IntegrationByPartsCProof}
		d_\s \int_{\B^+_R \cap \{ y > \delta \} } y^a \nabla \textbf{c} \cdot \nabla  (\textbf{c} \eta^2) \d x = - d_\s \int_{B_R} \delta^a \textbf{c}_y(x',\delta)   \textbf{c}  (x',\delta) \eta^2 (x',\delta) \d x'.
	\end{equation}
	Taking the limit $\delta \to 0$ and using that $ \textbf{c}$ and $y^a \textbf{c}_y$ are continuous up to $\{y=0\}$ (as proved before), from \eqref{Eq:EquationsForC} we conclude \eqref{Eq:IntegrationByPartsC}.

	Combining \eqref{Eq:IntegrationByPartsC} with \eqref{Eq:StabilityConditionCEtaWeak} we get
	$$
	-2\s \int_{B_1} \dfrac{f(u)}{d_\s}\textbf{c} \eta^2\d x' \leq \int_{\B_1^+} y^a \textbf{c}^2  |\nabla \eta |^2  \d x.
	$$
	Next, we rewrite the left-hand side of this inequality after an integration by parts (now using that $d_\s \partial_{\nu^a} v = f(u)$ in $B_1$ and that $v$ is $\s$-harmonic), obtaining
	\begin{equation}
		\label{Eq:Stabilityc=rv_rProof1}
		-2\s  \int_{\B_1^+} y^a (\nabla v \cdot \nabla \textbf{c} )\eta^2 \dx -2\s \int_{\B_1^+} y^a \textbf{c} \nabla v \cdot \nabla  (\eta^2)  \dx \leq   \int_{\B_1^+} y^a \textbf{c}^2  |\nabla \eta |^2  \d x .
	\end{equation}
	
	We treat the first integral in this inequality by using the same idea as in the proof of the Pohozaev identity, to get rid of second order derivatives of $v$. 
	Note first that 
	\begin{equation}
		\label{Eq:GradvGradc}
		\nabla v \cdot \nabla \textbf{c} = \sum_{j=1}^{n+1} v_j \partial_j  \sum_{i=1}^{n+1} x_i v_i  = \sum_{j=1}^{n+1} v_j^2 + \sum_{j=1}^{n+1} \sum_{i=1}^{n+1} x_i v_j  v_{ij}=  |\nabla v|^2 + x \cdot  \dfrac{\nabla|\nabla v|^2}{2}. 
	\end{equation}
	We now claim that
	\begin{equation}
		\label{Eq:IntegrationByPartsPohozaev}
		- \int_{\B_1^+ } y^a \left(x \cdot \nabla|\nabla v|^2\right) \eta^2 \d x =   (n + 2 - 2\s) \int_{\B_1^+ }y^a |\nabla v|^2 \eta^2 \d x  +  \int_{\B_1^+ } y^a  |\nabla v|^2 r (\eta^2)_r \d x.
	\end{equation}
	To see this, we proceed as before (integrating by parts on $\B^+_1 \cap \{ y > \delta \}$, where all the functions are smooth) to obtain
	\begin{equation}
		\begin{split}
			- \int_{\B_1^+ \cap \{ y > \delta \} } y^a \left(x \cdot \nabla|\nabla v|^2\right) \eta^2 \d x &= \int_{\B_1^+ \cap \{ y > \delta \}} \div(y^a x) |\nabla v|^2 \eta^2 \d x \\
			& \quad \quad  +   \int_{\B_1^+ \cap \{ y > \delta \} } y^a  |\nabla v|^2 r (\eta^2)_r \d x \\
			& \quad \quad + \int_{ B_1}  \delta^{a + 1} \big(|\nabla_{x'} v(x', \delta)|^2+ |v_y(x', \delta)|^2 \big) \eta^2(x', \delta) \d x'.
		\end{split}
	\end{equation}
	Now, using that $\nabla_{x'} v$ and $y^a v_y$ are continuous up to $\{y=0\}$ (locally in $B_1$, but recall that $\eta$ has compact support in $B_1$), we have
	$$
	\delta^{a + 1} \big(|\nabla_{x'} v(x', \delta)|^2+ |v_y(x', \delta)|^2 \big) \leq C \delta^{2 - 2\s} + C'\delta^{1-a}|\delta^a v_y (x', \delta)|^2 \leq C(\delta^{2 - 2\s} + \delta^{2\s}),
	$$
	and thus \eqref{Eq:IntegrationByPartsPohozaev} is obtained after letting $\delta\to 0$ and using that $\div(y^a x) = (n + 1 + a)y^a = (n + 2 - 2\s)y^a$.

	From \eqref{Eq:GradvGradc} and \eqref{Eq:IntegrationByPartsPohozaev} we get
	\begin{equation}
		\begin{split}
			- \int_{\B_1^+ } y^a (\nabla v \cdot \nabla \textbf{c})  \eta^2 \dx 
			&= - \int_{\B_1^+ } y^a |\nabla v|^2 \eta^2 \d x  - \dfrac{1}{2}\int_{\B_1^+ } y^a \left(x \cdot \nabla|\nabla v|^2\right) \eta^2 \d x  \\
			&= \dfrac{n-2\s}{2} \int_{\B_1^+ } y^a |\nabla v|^2 \eta^2 \d x +\dfrac{1}{2} \int_{\B_1^+ } y^a  |\nabla v|^2 r (\eta^2)_r \d x.
		\end{split}
	\end{equation}
	Combining this with \eqref{Eq:Stabilityc=rv_rProof1} we conclude the proof. 	
\end{proof}

\begin{remark}
	The second identity in \eqref{Eq:EquationsForC} can be established directly ``downstairs'', that is, without using the extension. 
	Indeed, since $\textbf{c}$ is the $\s$-harmonic extension of $ x' \cdot \nabla_{x'} u$, it holds $ d_\s \partial_{\nu^a} \textbf{c} = \fraclaplacian (x' \cdot \nabla_{x'} u)$. 
	Then, the second expression in \eqref{Eq:EquationsForC} follows from the identity
	$$
	\fraclaplacian (x' \cdot \nabla_{x'} w) = x' \cdot \nabla_{x'} \{ \fraclaplacian w \} + 2\s \fraclaplacian w,
	$$
	(which holds for all functions $w$) by taking $w=u$.
	The identity can be proved easily using the definition of the fractional Laplacian, or using that $x' \cdot \nabla_{x'} w (x') = \frac{\d}{\d h}|_{h=1} w(hx')$ and the scaling properties of $\fraclaplacian$.
\end{remark}

From the inequality of \Cref{Lemma:Stabilityc=rv_r} and choosing $\eta = r^{(2\s-n)/2} \zeta$ (properly regularized near the origin), with $\zeta$ being a cut-off function, we can establish \Cref{Prop:EstimatevrAnnulus}.

\begin{proof}[Proof of \Cref{Prop:EstimatevrAnnulus}]
	For $\varepsilon \in (0,1/2)$, we take the Lipschitz function
	$$
	\eta = \begin{cases}
		\varepsilon^{-\alpha/2} & \text{ in } \B_\varepsilon^+,\\
		r^{-\alpha/2} \zeta & \text{ in }  \B_1^+  \setminus \B_\varepsilon^+,
	\end{cases}
	$$
	where $0\leq \alpha \leq n-2\s$ is to be chosen later and $\zeta = \zeta(r)$ is a cut-off function satisfying $\zeta = 1$ for $r\leq 1/2$ and $\zeta = 0$ for $r\geq 3/4$.
	We use this choice of $\eta$ in \eqref{Eq:Stabilityc=rv_r}.  
	To let $\varepsilon \to 0$ in the resulting inequality, we use the dominated convergence theorem taking into account that $y^a |\nabla v|^2 = y^a (|\nabla_{x'} v|^2+ v_y^2)\leq C(y^{1 - 2\s} + y^{2\s - 1})$ in $\B^+_{3/4}$ ---by \Cref{Lemma:RegularityHorizontalGrad}~($b$)--- and that $-\alpha \geq 2\s - n$.
	We conclude that $I_1 + I_2 \leq I_3$, where
	$$ 
	I_1 := \s \int_{\B_1^+} y^a \left \{ (n - 2\s)r^{-\alpha}\zeta^2 + r (r^{-\alpha}\zeta^2)_r \right \} |\nabla v|^2 \dx,
	$$
	$$ 
	I_2 := -2 \s\int_{\B_1^+} y^a r v_r \nabla v \cdot \nabla (r^{-\alpha}\zeta^2 ) \d x, 
	$$
	and
	$$
	I_3 := \int_{\B_1^+} y^a r^2 v_r^2 |\nabla (r^{-\alpha/2}\zeta )|^2 \d x.
	$$
	
	Now, we compute each of the previous integrals separately, using that $\zeta \equiv 1$ in $\B_{1/2}^+$ ---and thus $\nabla \zeta \equiv 0$ in $\B_{1/2}^+$. First, note that 
	\begin{align}
		I_1 &= \s(n - 2\s - \alpha)\int_{\B_{1/2}^+} y^a r^{-\alpha}|\nabla v|^2 \d x \\
		& \quad \quad+  \s\int_{\B_{3/4}^+ \setminus \B_{1/2}^+}  y^a r^{-\alpha}|\nabla v|^2 \big ((n - 2\s - \alpha) \zeta^2 + r (\zeta^2)_r \big )\d x.
	\end{align}
	For $I_2$ we have
	\begin{align}
		I_2 &= 2\alpha \s \int_{\B_{1/2}^+} y^a r^{-\alpha} v_r^2  \d x + 2\s \int_{\B_{3/4}^+ \setminus \B_{1/2}^+} y^a r^{-\alpha} \big ( \alpha  v_r^2 \zeta^2 - r v_r \nabla v \cdot \nabla (\zeta^2)\big )\d x.
	\end{align}
	Last, since $	|\nabla (r^{-\alpha/2}\zeta )|^2 = (\alpha^2/4) r^{- \alpha - 2} \zeta^2 + r^{-\alpha} |\nabla \zeta|^2 - \alpha r^{-\alpha - 1} \zeta \zeta_r$, we get
	\begin{align}
		I_3 &= \dfrac{\alpha^2}{4} \int_{\B_{1/2}^+} y^a r^{- \alpha} v_r^2 \d x + \int_{\B_{3/4}^+ \setminus \B_{1/2}^+} y^a r^{-\alpha} v_r^2 \left ( \dfrac{\alpha^2}{4} \zeta^2 +  r^2|\nabla \zeta|^2 - \alpha r \zeta \zeta_r \right )\d x.
	\end{align}
	
	Next, we estimate the three previous integrals  over $\B_{3/4}^+ \setminus \B_{1/2}^+$, noticing that the integrands can be bounded by $C y^a |\nabla v|^2$, with $C$ depending only on $n$, $\s$, and $\alpha$.
	Hence, from the inequality $I_1 + I_2 \leq I_3$  we get
	$$
	\s (n - 2\s - \alpha)\int_{\B_{1/2}^+} y^a r^{-\alpha}|\nabla v|^2 \d x 
	+\alpha \bpar{2\s - \dfrac{\alpha}{4}} \int_{\B_{1/2}^+} y^a r^{-\alpha} v_r^2  \d x  \leq C \int_{\B_{3/4}^+ \setminus \B_{1/2}^+} y^a |\nabla v|^2\d x.	 
	$$
	
	Note that $\alpha (2\s - \alpha/4) > 0$ whenever $\alpha \in (0, 8\s)$. 
	Thus, if we take $\alpha = n - 2\s$, since the proposition assumes $n\in (2\s, 10\s)$ we finally obtain the desired estimate
	$$
	\int_{\B_{1/2}^+} y^a r^{2\s - n} v_r^2 \d x \leq C \int_{\B_{3/4}^+ \setminus \B_{1/2}^+} y^a |\nabla v|^2  \d x
	$$
	with a constant $C$ depending only on $n$ and $\s$.
\end{proof}

%%%%%%%%%%%%%%%%%%%%%%%%%%%%%%%%%%%%%%%%%%%%%%%%%%%%%%%%%%%%%
%%%%%%%%%%%%%%%%%%%%%%%%%%%%%%%%%%%%%%%%%%%%%%%%%%%%%%%%%%%%%
\section{A new geometric form of the stability condition}
%%%%%%%%%%%%%%%%%%%%%%%%%%%%%%%%%%%%%%%%%%%%%%%%%%%%%%%%%%%%%
%%%%%%%%%%%%%%%%%%%%%%%%%%%%%%%%%%%%%%%%%%%%%%%%%%%%%%%%%%%%%
\label{Sec:GeomStability}

In this section we establish \Cref{Th:GeomStabilityFullGradient_s=1/2}, which gives a new geometric form of the stability condition.
As a corollary, we will obtain an $L^2$ estimate for the full Hessian of the harmonic extension $v$ of a stable solution $u$.
This last result, stated in \Cref{Coro:W12Gradient} below, will be crucial in the following sections.

\begin{proof}[Proof of \Cref{Th:GeomStabilityFullGradient_s=1/2}]
	We take $\xi = \textbf{c} \eta$ in the stability condition \eqref{Eq:StabilityExtensionHalfLaplacian}, where $\eta$ is a Lipschitz function with compact support in $\Omega \times [0,+\infty)$ (as in the statement of the theorem) and $\textbf{c}$ is a function (to be chosen later) in $ C^1_\loc(\Omega\times [0,+\infty))$,  smooth in $\R^{n+1}_+$, and such that  $\textbf{c}\,\Delta \textbf{c} \in C_\loc(\Omega\times [0,+\infty))$.
	Notice that, as a consequence, $\xi \in H^1(\Omega\times [0,+\infty))$.
	We obtain
	\begin{equation}
		\label{Eq:StabilityConditionCEtaWeakHalf}
		\int_{\Omega} f'(u) \textbf{c}^2  \eta^2 \dx' \leq \int_{\R^{n+1}_+} \nabla \textbf{c} \cdot \nabla (\textbf{c} \eta^2) \d x + \int_{\R^{n+1}_+} \textbf{c}^2  |\nabla \eta |^2 \d x.
	\end{equation}
	We now integrate by parts the first term on the right-hand side.
	For the integration by parts to be fully justified, we perform it in $\{y > \delta\}$ (where $\textbf{c}$ is smooth) and then we let $\delta \to 0$, using dominated convergence and the continuity of $\textbf{c}\,\textbf{c}_y$ and $\textbf{c}\,\Delta \textbf{c}$.
	We get
	\begin{equation}
		\label{Eq:StabilityConditionCEta_s=1/2}
		\int_{\Omega} \left (\textbf{c}\,\textbf{c}_y +  f'(u) \textbf{c}^2 \right )  \eta^2 \d x' 
		\leq \int_{\R^{n+1}_+} ( \textbf{c}^2  |\nabla \eta |^2  - \textbf{c} \, \Delta \textbf{c} \, \eta^2 ) \d x .
	\end{equation}

	We wish to take $\textbf{c}$ to be 
	$ |\nabla v|$
	in the previous inequality.
	Nevertheless, since this function may not be $C^1$ (if $\nabla v$ vanishes somewhere), we consider instead the smooth function
	$$
	\textbf{c}_\varepsilon := (|\nabla v|^2  + \varepsilon^2)^{1/2},
	$$
	for $\varepsilon > 0$,
	which is locally $C^1$ in $\Omega \times [0,+\infty)$ since $v \in C^2_\loc (\Omega \times [0,+\infty))$ by  \Cref{Lemma:RegularityHorizontalGrad}~($c$).
	Moreover, an easy computation using that $v$ is harmonic shows that
	\begin{equation}
		\label{Eq:cLaplacec}
		\textbf{c}_\varepsilon \, \Delta \textbf{c}_\varepsilon = \sum_{i,j=1}^{n+1}  v_{ij}^2 -  \sum_{j=1}^{n+1}\left( \sum_{i=1}^{n+1} v_{ij} \dfrac{v_i }{\textbf{c}_\varepsilon}  \right)^2 
		\geq \acal^2 \quad \text{ in } \R^{n+1}_+.
	\end{equation}	
	In particular, $\textbf{c}_\varepsilon \, \Delta \textbf{c}_\varepsilon  \in C_\loc (\Omega\times [0,+\infty))$ ---since $v\in C^2_\loc (\Omega\times [0,+\infty))$.
	As a consequence, we can choose $\textbf{c}$ in \eqref{Eq:StabilityConditionCEta_s=1/2} to be $\textbf{c}_\varepsilon$ and conclude, taking into account  \eqref{Eq:cLaplacec},
	\begin{equation}
		\label{Eq:StabilityFullGradProof}
		\int_{\R^{n+1}_+}  \acal^2 \eta^2  \d x + \int_{\Omega} \left (\textbf{c}_\varepsilon\, \partial_y \textbf{c}_\varepsilon + f'(u) \textbf{c}_\varepsilon^2 \right ) \eta^2 \d x' 
		\leq \int_{\R^{n+1}_+}  \textbf{c}_\varepsilon^2  |\nabla \eta |^2  \d x .
	\end{equation}
	
	Next, we claim that 
	\begin{equation}
		\label{Eq:GeomStabProof1LHS}
		\textbf{c}_\varepsilon\, \partial_y \textbf{c}_\varepsilon + f'(u) \textbf{c}_\varepsilon^2  \geq \varepsilon^2 f'(u) \quad \text{ in } \Omega \subset \{y= 0\}.
	\end{equation}
	Once this is proved, the result will follow by letting $\varepsilon\to 0$.
	
	To establish \eqref{Eq:GeomStabProof1LHS}, we first compute, for $y>0$,
	\begin{equation}
		\label{Eq:GeomStabProof2LHS}
		\textbf{c}_\varepsilon\, \partial_y \textbf{c}_\varepsilon = \sum_{i=1}^{n+1} v_i \, v_{i y} = \nabla_{x'} v \cdot \nabla_{x'} v_y + v_y \, v_{yy}.
	\end{equation}
	Now, on the one hand note that  $-\partial_y \nabla_{x'}v = f'(u) \nabla_{x'}u$ in $\Omega $ by \Cref{Lemma:RegularityHorizontalGrad}~($b$). 
	Multiplying this identity by $\nabla_{x'}v$ we get
	$$
	\nabla_{x'} v \cdot \nabla_{x'} v_y = -f'(u)|\nabla_{x'} u|^2 \quad \text{ in } \Omega.
	$$
	On the other hand, since $v_{yy} = - \Delta_{x'}v$ and $\Delta_{x'}v$ is continuous up to $\{y= 0\}$, we see that
	$$
	\lim_{y\downarrow 0} v_{yy} = - \Delta_{x'}u = (-\Delta_{x'})^{1/2}  (-\Delta_{x'})^{1/2} u =  (-\Delta_{x'})^{1/2} f(u)\quad \text{ in } \Omega. 
	$$
	Using the previous identities together with $-v_y = f(u)$ in $\Omega$,  \eqref{Eq:GeomStabProof2LHS} leads to
	$$
	\textbf{c}_\varepsilon\, \partial_y \textbf{c}_\varepsilon =  -f'(u)|\nabla_{x'} u|^2 -f(u) (-\Delta_{x'})^{1/2} f(u)\quad \text{ in } \Omega. 
	$$
	As a consequence, noticing that $f'(u) \textbf{c}_\varepsilon^2 = f'(u) |\nabla_{x'} u|^2 + f'(u) f(u)^2 + \varepsilon^2 f'(u)$ in $\Omega$, we have
	\begin{align}
		\textbf{c}_\varepsilon\, \partial_y \textbf{c}_\varepsilon + f'(u) \textbf{c}_\varepsilon^2 &=   -f(u) (-\Delta_{x'})^{1/2} f(u) + f'(u) f(u)^2  + \varepsilon^2 f'(u)\\
		& = f(u) \left \{ f'(u) (-\Delta_{x'})^{1/2} u - (-\Delta_{x'})^{1/2} f(u) \right\}  + \varepsilon^2 f'(u).
	\end{align}
	Finally, using that $f(t_1)-f(t_2) \leq f'(t_1) (t_1- t_2)$ for all $t_1,t_2 \in \R$ (since $f$ is convex), we see that $(-\Delta_{x'})^{1/2} f(u) \leq f'(u)( -\Delta_{x'})^{1/2} u$.
	Hence, since $f$ is nonnegative, the claim \eqref{Eq:GeomStabProof1LHS} follows.
\end{proof}

\begin{remark}
	\label{Remark:s=1/2}
	In the previous proof, we have used strongly that $\s=1/2$, both in the choice of the test function $\xi = |\nabla v |\eta$ and when using that $-\Delta = \halflaplacian \circ \halflaplacian$.
	It is not clear to us how to extend the previous arguments to other powers $\s \in (0,1)$.		
	On the one hand,  $|\nabla v|$ is not an appropriate test function anymore, at least when $\s <1/2$ (indeed, since $v_y$ behaves as $-y^{2\s - 1}f(v)$ near $\{y=0\}$, $|\nabla v|$ may be singular at almost all points in $\{y=0\}$).
	On the other hand, one would say that the analogue decomposition of $-\Delta$ for powers $\s\neq 1/2$ should be $-\Delta = \fraclaplacian \circ (-\Delta)^{1-\s}$, but its is not clear which information  on $(-\Delta)^{1-\s}$ is to be used.
\end{remark}

From \Cref{Th:GeomStabilityFullGradient_s=1/2} we can now obtain an $H^1$ estimate for the gradient of the harmonic extension of a stable solution.
The estimate is independent of the nonlinearity $f$, a fact that will be crucial in the following sections.
Note that the norms are now computed on the intersection of $\R^{n+1}_+$ with balls not necessarily centered at $\partial\R^{n+1}_+$, since this will be useful later on.

\begin{corollary}
	\label{Coro:W12Gradient}
	Let $n \geq 1$ and let $u\in C^2(B_1)\cap L^1_{1/2}(\R^n)$ be a stable solution to $(-\Delta)^{1/2}u = f(u)$ in $B_1 \subset \R^n$, with $f$ a nonnegative convex $C^{1,\gamma}$ function for some $\gamma >0$. 
	Let $v$ be the harmonic extension of $u$ in $\R^{n+1}_+$.

	Then, for every $0 < R_1 < R_2 \leq 1$ and every $x_0 \in \B_1$ such that $ \B_{R_2}(x_0) \subset \B_1$, 
	\begin{equation}
		\label{Eq:W12Gradient}
		\norm{\nabla v }_{L^2(\B_{R_1}(x_0)  \cap \{y>0\})} 
		+(R_2 - R_1)\norm{ D^2 v }_{L^2(\B_{R_1}(x_0)  \cap \{y>0\})} \leq C \norm{ \nabla v }_{L^2 (\B_{R_2}(x_0)  \cap \{y>0\})}
	\end{equation}
	for some positive dimensional constant $C$.
\end{corollary}

\begin{proof}
	The bound for the $L^2$ norm of $\nabla v$ in $\B_{R_1}(x_0) \cap \{y>0\}$ is trivial; thus we only need to prove an $L^2$ estimate for $D^2v$.
	To carry this out, choose a cut-off function $\eta$ with compact support in $\B_{R_2}(x_0)$ such that $\eta \equiv 1$ in $\B_{R_1}(x_0)$ ---hence its gradient is supported in $\B_{R_2}(x_0) \setminus \B_{R_1}(x_0)$ and, in this set, $|\nabla \eta|\leq C/(R_2-R_1)$ for some dimensional constant $C$.
	Now, since $\B_{R_2}(x_0) \subset \B_1$, we can use  \Cref{Th:GeomStabilityFullGradient_s=1/2} with this choice of $\eta$ to obtain
	$$
	\int_{\B_{R_1}(x_0) \cap \{y>0\} } \acal^2 \d x \leq \dfrac{C}{(R_2-R_1)^2} \int_{(\B_{R_2}(x_0)\setminus \B_{R_1}(x_0))  \cap \{y>0\}} |\nabla v|^2 \d x.
	$$
	The corollary will follow once we show that the square of every element of $D^2 v$ can be controlled pointwise by $C\acal^2$ for some dimensional constant $C$.
	
	Let us prove this last assertion.
	At every point $x\in \R^{n+1}_+ \cap \{ \nabla v \neq 0\}$ we  consider the normal unitary vector to the level set of $v$ passing through $x$,
	$$
	\nu := \dfrac{\nabla v(x)}{|\nabla v(x)|},
	$$
	and any orthonormal basis $\{\tau_1, \ldots, \tau_n, \nu\}$ of $\R^{n+1}$.
	Expressing $D^2 v(x)$ in this basis, we see that
	$$
	\acal^2 (x) = \sum_{i,j=1}^n (D^2 v (x) [\tau_i, \tau_j])^2 + \sum_{i=1}^n (D^2 v (x) [\tau_i, \nu])^2;
	$$
	recall that $\acal^2$ is defined in \eqref{Eq:AcalDef}.
	Now,  by the symmetry of $D^2 v$ it is clear that $\acal^2(x)$ controls the square of every entry in $D^2 v (x) $  except for $D^2 v(x)[\nu, \nu]$.
	However, since $v$ is harmonic, we have that 
	$$
	D^2 v (x)[\nu, \nu] = -\sum_{i=1}^n D^2 v (x)[\tau_i, \tau_i],
	$$ 
	and thus we have the desired control at every point $x\in \R^{n+1}_+ \cap \{ \nabla v \neq 0\}$.
	
	Finally, by noticing that $D^2 v = 0$ a.e. in $\R^{n+1}_+ \cap \{ \nabla v = 0\}$ since $\nabla v$ is Lipschitz (see for instance~\cite[Theorem~6.19]{LiebLoss}), we conclude the proof.
\end{proof}

%%%%%%%%%%%%%%%%%%%%%%%%%%%%%%%%%%%%%%%%%%%%%%%%%%%%%%%%%%%%%
%%%%%%%%%%%%%%%%%%%%%%%%%%%%%%%%%%%%%%%%%%%%%%%%%%%%%%%%%%%%%
\section{Decay of the weighted radial derivative}
%%%%%%%%%%%%%%%%%%%%%%%%%%%%%%%%%%%%%%%%%%%%%%%%%%%%%%%%%%%%%
%%%%%%%%%%%%%%%%%%%%%%%%%%%%%%%%%%%%%%%%%%%%%%%%%%%%%%%%%%%%%
\label{Sec:GeometricDecayGradient}

In this section, we establish the geometric decay in $R$ of a weighted $L^2$ norm of $v_r$ in~$\B_R^+$, where $v$ is the harmonic extension of a stable solution.
The main ingredient is the following key lemma.
It establishes that, under a doubling assumption on $ |\nabla v|^2 \d x $, the full gradient of $v$ is controlled, in $L^2$, by the radial derivative of~$v$ in an annulus.
To prove this result we use, as in~\cite{CabreFigalliRosSerra-Dim9}, a compactness argument combined with the nonnegativeness of the nonlinearity~$f$ and the nonexistence of nonconstant 0-homogeneous superharmonic functions.

\begin{lemma}
	\label{Lemma:v_rControlsGradientInAnnulus}
	Let $n \geq 1$ and let $u\in C^2(B_2)\cap L^1_{1/2}(\R^n)$ be a stable solution to $(-\Delta)^{1/2}u = f(u)$ in $B_2 \subset \R^n$, with $f$ a nonnegative convex $C^{1,\gamma}$ function for some $\gamma >0$. 
	Let $v$ be the harmonic extension of $u$ in $\R^{n+1}_+$. 
	Assume that
	$$
	\int_{\B_{1}^+}  |\nabla v|^2 \d x \geq \delta \int_{\B_2^+}  |\nabla v|^2  \d x
	$$
	for some constant $\delta>0$. 
	
	Then, 
	$$
	\int_{\B_{3/2}^+}  |\nabla v|^2  \d x \leq C_\delta \int_{\B_{3/2}^+ \setminus \B_{1}^+} |x\cdot \nabla v|^2  \d x
	$$
	for some constant $C_\delta$ depending only on $n$ and $\delta$.
\end{lemma}

\begin{proof}%[Proof of \Cref{Lemma:v_rControlsGradientInAnnulus}]
	By contradiction, assume the result to be false. 
	Then, there exists a sequence of stable solutions $u_k$ in $B_2$ such that their harmonic extensions $v_k$ satisfy
	\begin{equation}
		\label{Eq:DoublingAssumption}
		\int_{\B_{1}^+}  |\nabla v_k|^2 \d x \geq \delta \int_{\B_2^+}  |\nabla v_k|^2  \d x,
	\end{equation}
	\begin{equation}
		\label{Eq:GradientNormalizedL2}
		\int_{\B_{3/2}^+}  |\nabla v_k|^2  \d x = 1,
	\end{equation}
	and
	$$
	\int_{\B_{3/2}^+ \setminus \B_{1}^+}   |x\cdot\nabla v_k|^2  \d x \to 0 \quad \text{ as } k \to \infty.
	$$
	Note that each $u_k = v_k(\cdot, 0) $ solves a different equation, $(-\Delta)^{1/2}u_k = f_k(u_k)$ in $B_2$, for some convex nonlinearities $f_k\geq0$.
	To ensure \eqref{Eq:GradientNormalizedL2}, we use that the class of solutions considered in the lemma is invariant under multiplication by constants.

	The crucial point in the following arguments is that the estimate from \Cref{Coro:W12Gradient} does not depend on the nonlinearity $f_k$.
	Thus, we will be able to use it to obtain uniform estimates in $k$.
	The details go as follows.
	
	Using \Cref{Coro:W12Gradient} (rescaled to hold in $B_2$ and with $x_0 = 0$ and $R_1 = 3/2$ and $R_2 = 2$ after the rescaling), the doubling assumption~\eqref{Eq:DoublingAssumption}, and \eqref{Eq:GradientNormalizedL2}, we get
	$$
	\norm{ \nabla v_k }^2_{H^1(\B_{3/2}^+)} 
	\leq C \norm{ \nabla v_k }^2_{L^2 (\B_2^+)}
	\leq \dfrac{C}{\delta} \norm{ \nabla v_k }^2_{L^2 (\B_1^+)} \leq \dfrac{C}{\delta},
	$$
	where $C$ is a positive dimensional constant.	
	Therefore, by Rellich's compactness theorem, up to removing a constant from $v_k$, a subsequence of $v_k$ converges in $H^1(\B_{3/2}^+)$ to some function~$v_\infty$. 
	Moreover,
	\begin{equation}
		\label{Eq:Gradient=1}
		\int_{\B_{3/2}^+}  |\nabla v_\infty|^2  \d x = 1
	\end{equation}
	and 
	$$
	\int_{\B_{3/2}^+ \setminus \B_{1}^+}  |x\cdot \nabla v_\infty|^2  \d x = 0.
	$$
	This yields that $v_\infty$ is a $0$-homogeneous function in $\B_{3/2}^+ \setminus \B_{1}^+$.
	
	Notice that each $v_k$ is harmonic in $\B_{3/2}^+$.
	In addition, $v_k$ has nonnegative flux on $B_{3/2} \subset \{y=0\}$ in the weak $H^1$ sense, i.e.,
	\begin{equation}
		\int_{\B_{3/2}^+} \nabla v_k \cdot \nabla \varphi \d x \geq 0
	\end{equation} 
	for every nonnegative smooth function $\varphi$ with compact support in $\B_{3/2}^+\cup B_{3/2}$.
	By the $H^1$ convergence the same two properties hold for $v_\infty$. 
	Therefore, by considering the even reflection of $v_\infty$ across  $\partial \R^{n+1}_+$ we obtain a function, still denoted  by $v_\infty$, which is  $0$-homogeneous in $\B_{3/2} \setminus \B_{1}$ and weakly superharmonic in $\B_{3/2}$.
	
	As a consequence of the mean value property for superharmonic functions, $v_\infty$ is bounded below in $\partial \B_{5/4}$  and thus, by $0$-homogeneity, also in $\B_{3/2} \setminus \B_{1}$.
	Moreover,  $\inf_{\B_{3/2} \setminus \B_{1}} v_\infty = \inf_{\B_{1/8}(x_0)} v_\infty$ for some $x_0\in \partial \B_{5/4}$.
	Therefore, the strong maximum principle (\cite[Theorem~8.19]{GilbargTrudinger}) yields that $v_\infty$ is constant in $\B_{3/2} \setminus \B_{1}$, say $v_\infty \equiv c_0$.
	
	Finally, since ${v_\infty}_{|\partial \B_1} = c_0$, by superharmonicity it follows that $v_\infty \geq c_0$ in $\B_{3/2}$. 
	But since $v_\infty \equiv c_0$ in $\B_{3/2} \setminus \B_{1}$, the strong maximum principle for superharmonic functions yields that $v_\infty \equiv c_0$ in $\B_{3/2}$.	
	This contradicts \eqref{Eq:Gradient=1} and concludes the proof.
\end{proof}

Thanks to the control of the gradient in terms of the radial derivative (under a doubling assumption), given by the previous lemma, we can establish the estimate that will lead to Hölder continuity.
Note that here we need to assume $2 \leq n \leq 4$.

\begin{proposition}
	\label{Prop:GeometricDecayOfDirichletInt}
	Let $u\in C^2(B_1)\cap L^1_{1/2}(\R^n)$ be a stable solution to $(-\Delta)^{1/2}u = f(u)$ in $B_1 \subset \R^n$, with $f$ a nonnegative convex $C^{1,\gamma}$ function for some $\gamma >0$. 
	Let $v$ be the harmonic extension of $u$ in $\R^{n+1}_+$.
	
	If $2 \leq n \leq 4$, then
	$$
	\int_{\B_{R}^+} r^{1-n} v_r^2 \d x  \leq C  R^{2\alpha} \int_{\B_{3/4}^+}  |\nabla v|^2  \d x  \quad \text{ for } R \leq 1/2,
	$$
	where $\alpha$ and $C$ are positive dimensional constants.
\end{proposition}

\begin{proof} %[Proof of \Cref{Prop:GeometricDecayOfDirichletInt}]
	
	For $j\geq 0$, define
	$$
	a_j:= 2^{-j(1-n)} \int_{\B_{1/2^{j+1}}^+}  |\nabla v|^2  \d x  \quad \text{ and } \quad 
	b_j:=  \int_{\B_{1/2^{j+1}}^+}  r^{1 - n} v_r^2 \d x.
	$$
	By \Cref{Prop:EstimatevrAnnulus} (used with $\s= 1/2$) there exists a dimensional constant $C$ for which $b_0\leq C \norm{\nabla v}^2_{L^2(\B^+_{3/4})}$, and thus
	\begin{equation}
		\label{Eq:Recurence1}
		a_0 \leq M \quad \text{ and } \quad b_0 \leq M , \quad \text{ where } \quad M:= \max \{1,C\}\int_{\B_{3/4}^+}  |\nabla v|^2  \d x.
	\end{equation}
	
	Clearly, $b_j \leq b_{j-1}$ for $j\geq 1$.
	Moreover, by applying \Cref{Prop:EstimatevrAnnulus} to $v(\cdot /2^{j})$, it follows\footnote{It is useful to note that $a_j$ and $b_j$ are adimensional quantities, and that \Cref{Prop:EstimatevrAnnulus} can be written equivalently having the weight $r^{2\s - n}$ inside the integral of the right-hand side. } that 
	\begin{equation}
		\label{Eq:Recurence3prev}
		a_j + b_j \leq L_1 a_{j-1}  \quad \text{ for } j \geq 1,
	\end{equation}
	for some positive dimensional constant $L_1$.
	Furthermore, by applying \Cref{Lemma:v_rControlsGradientInAnnulus} (with  $\delta = 2^{-n}$) to  $v(\cdot /2^{j+1})$ with $j \geq 1$, we get that if 
	\begin{equation}
		2^{-j(1-n)} \int_{\B_{1/2^{j+1}}^+}  |\nabla v|^2  \d x \geq 2^{-n} 2^{-j(1-n)} \int_{\B_{1/2^{j}}^+}  |\nabla v|^2  \d x = \dfrac{1}{2} \left ( 2^{-(j-1)(1-n)} \int_{\B_{1/2^{j}}^+}  |\nabla v|^2  \d x \right),
	\end{equation}
	then
	\begin{equation}
		\begin{split}
			2^{-j(1-n)} \int_{\B_{1/2^{j+1}}^+} \! \! |\nabla v|^2  \d x &\leq 	2^{-j(1-n)} \int_{\B_{3/2^{j+2}}^+} \!\! |\nabla v|^2  \d x \\
			&
			\leq C \int_{\B_{3/2^{j+2}}^+ \setminus \B_{1/2^{j+1}}^+} \!\! r^{1-n} v_r^2  \d x \leq C \int_{\B_{1/2^{j}}^+ \setminus \B_{1/2^{j+1}}^+} \!\! r^{1-n} v_r^2  \d x  
		\end{split}
	\end{equation}
	for some dimensional constant $C$.	
	That is, if $j\geq 1$ and $a_j \geq (1/2) a_{j-1}$, then $a_{j} \leq L_2 (b_{j-1} - b_{j})$ for some other dimensional constant $L_2$.
	By \eqref{Eq:Recurence3prev}, this leads to $ b_{j} \leq 2 L_1 L_2 (b_{j-1} - b_{j})$ provided that $a_j \geq (1/2) a_{j-1}$.
	
	Summarizing, there exist a dimensional constant $L$ such that
	\begin{equation}
		\label{Eq:Recurence3}
		b_j \leq b_{j-1} \quad \text{ and } \quad a_j + b_j \leq L a_{j-1}  \quad \text{ for } j \geq 1,
	\end{equation}
	and
	\begin{equation}
		\label{Eq:Recurence4}
		\text{ if } \quad  a_j \geq \dfrac{1}{2} a_{j-1} \quad \text{ then } \quad  b_{j} \leq L (b_{j-1} - b_{j})  \quad \text{ for } j \geq 1.
	\end{equation}
	Hence, from Lemma~3.2 of~\cite{CabreFigalliRosSerra-Dim9} if follows that there exist two dimensional constants $\theta \in (0,1)$ and $C_0>0$ for which
	\begin{equation}
		\label{Eq:GeometricDecay}
		b_j \leq C_0 \theta^j M \quad \text{ for } j \geq 0.
	\end{equation}
	Taking $\alpha$ such that $(1/2)^{2\alpha} = \theta$ we conclude the proof.
\end{proof}

\begin{remark}
	Although the decay of the weighted radial derivative given by \Cref{Prop:GeometricDecayOfDirichletInt} will suffice to establish our main result, we note that, with the current tools, the same decay can be proved for the full gradient.
	That is,
	$$
	\int_{\B_{R}^+} r^{1-n} |\nabla v|^2 \d x  \leq C  R^{2\alpha} M \quad \text{ for } R \leq 1/2,
	$$
	where $M$ is defined as in the previous proof.
	For this, one first proves by induction that there exists a dimensional constant $C_\star> 0$ such that 
	\begin{equation}
		\label{Eq:DyadicIntegralGeomDecay}
		d_k := \int_{\B_{1/2^{k-1}}^+ \setminus \B_{1/2^{k}}^+ } r^{1-n} |\nabla v|^2 \d x \leq C_\star \theta^k M \quad \text{ for } k\geq 2.
	\end{equation}
	From this, one concludes that 
	$$
	\int_{\B_{ 1/2^{j}}^+} r^{1-n} |\nabla v |^2 \d x  = \sum_{k=j+1}^\infty \int_{\B_{1/2^{k-1}}^+ \setminus \B_{1/2^{k}}^+ } r^{1-n} |\nabla v|^2 \d x
	\leq C_\star \sum_{k=j+1}^\infty \theta^k M= C \theta^j M
	$$
	for some dimensional constant $C$.
	
	To show \eqref{Eq:DyadicIntegralGeomDecay}, in each induction step we have the following dichotomy: either $d_{k+1} \leq \theta d_k$ (and  then \eqref{Eq:DyadicIntegralGeomDecay} follows readily from the induction hypothesis $d_k\leq C_\star \theta^k M$), or $d_{k+1} > \theta d_k$.
	In this second case, a suitable doubling assumption for $|\nabla v|^2 \d x$ is at our disposal. 
	This allows us to use \Cref{Lemma:v_rControlsGradientInAnnulus} (after an appropriate scaling) and hence control the $L^2$ norm of the gradient in the annulus by the $L^2$ norm of the radial derivative.
	From this and the decay given by \eqref{Eq:GeometricDecay}, we deduce \eqref{Eq:DyadicIntegralGeomDecay}.
\end{remark}

%%%%%%%%%%%%%%%%%%%%%%%%%%%%%%%%%%%%%%%%%%%%%%%%%%%%%%%%%%%%%
%%%%%%%%%%%%%%%%%%%%%%%%%%%%%%%%%%%%%%%%%%%%%%%%%%%%%%%%%%%%%
\section{$H^1$ control of the harmonic extension}
%%%%%%%%%%%%%%%%%%%%%%%%%%%%%%%%%%%%%%%%%%%%%%%%%%%%%%%%%%%%%
%%%%%%%%%%%%%%%%%%%%%%%%%%%%%%%%%%%%%%%%%%%%%%%%%%%%%%%%%%%%%
\label{Sec:H1Control}

In this section we establish a further ingredient towards the proof of \Cref{Th:Holder}.
In view of our last result, \Cref{Prop:GeometricDecayOfDirichletInt}, we must control the $H^1$ norm of the harmonic extension $v$ of a stable solution $u$ by the $L^1_{1/2}$ norm of $u$ in $\R^n$.
This is the content of the next proposition.

\begin{proposition}
	\label{Prop:H1ControlledByL1}
	Let $n \geq 1$ and let $u\in C^2(B_1)\cap L^1_{1/2}(\R^n)$ be a stable solution to $(-\Delta)^{1/2}u = f(u)$ in $B_1 \subset \R^n$, where $f$ is a nonnegative convex $C^{1,\gamma}$ function for some $\gamma >0$.
	Let $v$ be the harmonic extension of $u$ in $\R^{n+1}_+$. 
	
	Then, 
	\begin{equation}
		\label{Eq:H^1Estimate}
		\norm{ \nabla v}_{L^2(\B_{1/2}^+)}  \leq C \norm{u}_{L^1_{1/2}(\R^n)}
	\end{equation}	
	for some dimensional constant $C$.
	
	As a consequence,
	\begin{equation}
		\label{Eq:H^1/2Estimate}
		\seminorm{ u}_{H^{1/2}(B_{1/2})}  \leq C \norm{u}_{L^1_{1/2}(\R^n)}
	\end{equation}
	for some other dimensional constant $C$.
\end{proposition}

\begin{proof}
	The main idea of the proof is, as in \cite{CabreQuantitative}, to use the interpolation results of \Cref{Sec:Interpolation} combined with the $L^2$ estimate for $D^2v$ from \Cref{Coro:W12Gradient}.
	Since some error terms will appear in the right-hand side of the estimates, we will use an abstract lemma of L.~Simon to absorb them into the left-hand side. 
	In order to use this lemma, we establish our estimates in generic balls $\B_R (x_0) \subset \B_1$ (not necessarily half-balls centered at points $x_0 = (x_0', 0) \in \{y=0\}$) after intersecting them with $\R^{n+1}_+$.
	Through the proof we will use the letter $C$ to denote a dimensional constant which may change in each appearance.
	
	Given $x_0\in \B_1\subset \R^{n+1}$ and $R\in (0,1]$ such that $\B_R (x_0) \subset \B_1$, we cover the set $\B_{R/2} (x_0)  \cap \{y>0\}$ (up to a set of measure zero) with a family of disjoint open cubes $Q_j \subset \R^{n+1}_+$ of side-length $l_nR$, for some dimensional number $l_n$ small enough such that $Q_j \subset \B_{3R/4} (x_0)$. 
	
	Now, given $\varepsilon \in (0,1)$, in each cube we use the interpolation results of  \Cref{Prop:InterpolationGradient,Prop:InterpolationNash} (with $\R^n$ replaced by $\R^{n+1}$), properly rescaled to hold in the cubes $Q_j$,  and taking $\tilde{\ep} = \ep^{3/2}$ in \Cref{Prop:InterpolationNash}.
	Note that $v$ is $C^2$ in each $\overline{Q_j}$ by \Cref{Lemma:RegularityHorizontalGrad}~($c$).
	We obtain
	\begin{equation}
		\begin{split}
			R^{2-(n + 1)}\int_{Q_j} |\nabla v |^2 \d x & \leq C R^{3 -(n + 1)} \varepsilon  \int_{Q_j} |\nabla v | |D^2 v| \d x + C R^{2-(n + 1)} \varepsilon  \int_{Q_j} |\nabla v |^2 \d x \\
			& \quad \quad + C R^{-2(n + 1)} \varepsilon^{-2 - 3(n+1)/2} \left( \int_{Q_j} |v| \d x \right)^2
		\end{split}
	\end{equation}
	for every $\varepsilon\in (0,1)$.
	Multiplying the above inequality by $R^{2(n+1)}$ and using $2 R^3 |\nabla v | |D^2 v| \leq R^2 |\nabla v |^2 +  R^4 |D^2 v|^2$, we get
	\begin{equation}
		\begin{split}
			R^{n+3}\int_{Q_j} |\nabla v |^2 \d x & \leq C R^{n + 5} \varepsilon  \int_{Q_j}  |D^2 v|^2 \d x + C R^{n + 3} \varepsilon  \int_{Q_j} |\nabla v |^2 \d x \\
			& \quad \quad + C \varepsilon^{-2 - 3(n+1)/2} \left( \int_{Q_j} |v| \d x \right)^2.
		\end{split}
	\end{equation}
	Adding up all these inequalities (using that the disjoint cubes $Q_j$ cover $\B_{R/2}(x_0) \cap \{y>0\}$ and are contained in $\B_{3R/4}(x_0) \cap \{y>0\}$), we obtain
	\begin{equation}
		\begin{split}
			R^{n+3}\int_{\B_{R/2}(x_0) \cap \{y>0\}} |\nabla v |^2 \d x 
			& \leq C R^{n + 5} \varepsilon  \int_{\B_{3R/4}(x_0) \cap \{y>0\}}  |D^2 v|^2 \d x \\
			& \quad \quad + C R^{n + 3} \varepsilon  \int_{\B_{3R/4}(x_0) \cap \{y>0\}} |\nabla v |^2 \d x \\
			& \quad \quad + C \varepsilon^{-2 - 3(n+1)/2} \left( \int_{\B_{3R/4}(x_0) \cap \{y>0\}} |v| \d x \right)^2.
		\end{split}
	\end{equation}
	
	Next, we combine this information with \Cref{Coro:W12Gradient} (used with $R_1=3R/4$ and $R_2=R$), which gives that
	\begin{equation}
		R^2   \int_{\B_{3R/4}(x_0) \cap \{y>0\}}  |D^2 v|^2 \d x \leq C \int_{\B_{R}(x_0) \cap \{y>0\}} |\nabla v |^2 \d x.
	\end{equation}
	We deduce that
	\begin{equation}
		\label{Eq:H^1EstimateEpsilonBalls}
		\begin{split}
			R^{n+3}\int_{\B_{R/2}(x_0) \cap \{y>0\}} |\nabla v |^2 \d x 
			& \leq C R^{n + 3} \varepsilon  \int_{\B_{R}(x_0) \cap \{y>0\}} |\nabla v |^2 \d x \\
			&\quad \quad + C \varepsilon^{-2 - 3(n+1)/2} \left( \int_{\B_{1}^+} |v| \d x \right)^2.
		\end{split}
	\end{equation}

	Applying now an abstract result of L.~Simon ---that we use as stated in  \cite[Lemma~A.4]{CabreFigalliRosSerra-Dim9} taking $\beta = n + 3$ and $\sigma (\B_R (x_0)) := \norm{\nabla v}_{L^2 (\B_R (x_0)\cap \{y>0\})}^2$, after choosing $\varepsilon$ dimensionally small in \eqref{Eq:H^1EstimateEpsilonBalls}--- we get
	\begin{equation}
		\label{Eq:H1ControledByL1Ext}
		\int_{\B_{ 1/2}^+}  |\nabla v |^2 \d x \leq C \left( \int_{\B_{1}^+} |v | \d  x \right)^2.
	\end{equation}
	Finally, to bound the $L^1$ norm of $v$ in terms of its trace we use \Cref{Lemma:LpExtension}, which gives
	\begin{equation}
		\label{Eq:L1ExtLs}
		\int_{\B_{1}^+} |v| \d x \leq C  \norm{u}_{L^1_{1/2}(\R^n)}
	\end{equation}
	for some dimensional constant $C$.
	This concludes the proof of \eqref{Eq:H^1Estimate}.
	
	To establish the second estimate of the proposition, we take a cut-off function $\zeta$ such that $\zeta \equiv 1$ in $\B_{1/8}^+$ and $\zeta \equiv 0$ in $\R^{n + 1}_+ \setminus \B_{ 1/4}^+$. 
	Now, by the well-known trace inequality $\seminorm{w(\cdot, 0)}_{H^{1/2}(\R^n)} \leq \seminorm{w}_{H^{1}(\R^{n+1}_+)}$ used with $w = v \zeta$, we obtain
	\begin{align}
		\seminorm{u}_{H^{1/2}(B_{1/8})}^2 \leq \seminorm{u\zeta}_{H^{1/2}(\R^n)}^2 \leq \int_{\B_{ 1/4}^+} |\nabla (v \zeta)|^2 \d x \leq C  \left( \int_{\B_{ 1/4}^+} |\nabla v|^2 + \int_{\B_{ 1/4}^+} |v|^2 \right).
	\end{align}
	Using \Cref{Prop:InterpolationNash} as before (in a finite collection of disjoint cubes $Q_j$ covering $\B_{ 1/4}^+$ but all contained in $\B_{ 1/2}^+$, with $\tilde{\varepsilon} = 1/2$, and then adding up the resulting inequalities), we get that
	$$
	\int_{\B_{ 1/4}^+} |v|^2 \leq C \int_{\B_{ 1/2}^+} |\nabla v|^2 + C \left( \int_{\B_{ 1/2}^+} |v| \right)^2,
	$$ 
	and therefore
	$$
	\seminorm{u}_{H^{1/2}(B_{1/8})}^2 \leq C \int_{\B_{ 1/2}^+} |\nabla v|^2 + C \left( \int_{\B_{ 1/2}^+} |v| \right)^2.
	$$ 
	To conclude, we use the already proved estimate \eqref{Eq:H^1Estimate} together with \eqref{Eq:L1ExtLs} to obtain
	$$
	\seminorm{u}_{H^{1/2}(B_{1/8})} \leq C \norm{u}_{L^1_{1/2}(\R^n)}.
	$$ 
	
	Finally, to replace $B_{1/8}$ by $B_{1/2}$ in the above inequality and thus conclude \eqref{Eq:H^1/2Estimate}, we use a standard covering and scaling argument, noticing that all the estimates are independent of the nonlinearity $f$ and hence can be applied to rescaled stable solutions.
\end{proof}

%%%%%%%%%%%%%%%%%%%%%%%%%%%%%%%%%%%%%%%%%%%%%%%%%%%%%%%%%%%%%
%%%%%%%%%%%%%%%%%%%%%%%%%%%%%%%%%%%%%%%%%%%%%%%%%%%%%%%%%%%%%
\section{Proof of the main results}
%%%%%%%%%%%%%%%%%%%%%%%%%%%%%%%%%%%%%%%%%%%%%%%%%%%%%%%%%%%%%
%%%%%%%%%%%%%%%%%%%%%%%%%%%%%%%%%%%%%%%%%%%%%%%%%%%%%%%%%%%%%
\label{Sec:LinftyHalfLaplacian}

In this last section we finally establish \Cref{Th:Holder} as well as \Cref{Coro:LinftyConvexDiriclet,Coro:Extremal}.
In the case $n=1$ we borrow the results from next section.

To establish the Hölder bound of \Cref{Th:Holder}, we will use a Morrey-type estimate from~\cite{CabreQuantitative} adapted to radial derivatives. It is stated in  \Cref{Lemma:MorreyAverage} below. Notice that in the current boundary setting, to obtain Hölder regularity on $\R^n= \partial \R^{n+1}_+$ it will suffice to have a geometric-decay bound for radial derivatives with respect to points on $\R^n=\partial \R^{n+1}_+$. That is, we do not need control on the full gradient, neither on radial derivatives with respect to points in~$\R^{n+1}_+$.

Here and through the section, given a function $w$ defined in $\R^{n+1}_+$ and a point $z'\in \R^n$, we will denote the radius with respect to $z'$ by
\begin{equation}
	r_{z'} = 	r_{z'} (x) := |x - (z',0)|, \quad x \in \R^{n+1}_+,
\end{equation}
and the radial derivative of $w$ with respect to $z'$ by
\begin{equation}
	w_{r_{z'}} (x) := \dfrac{ x - (z',0)}{|x - (z',0)|} \cdot \nabla w(x), \quad x \in \R^{n+1}_+.
\end{equation}

\begin{lemma} [\cite{CabreQuantitative}]
	\label{Lemma:MorreyAverage}
	Let $z' \in \R^n$, $d>0$, and $w$ be a $C^1$ function in $\overline{\B^+_d}(z') \subset \overline{\R^{n+1}_+}$.
	Assume that, for some positive constants $\alpha$ and $C_1$, 
	\begin{equation}
		\label{Eq:MorreyGrowth}
		\int_{ \B_{R}^+ (z')} |w_{r_{z'}}| \d x \leq C_1 R^{n + \alpha} \quad \text{ for all } R \leq d.
	\end{equation}
	
	Let $S$ be any measurable subset of $\B^+_d(z')$ and $w_S:= \frac{1}{|S|}\int_S w \d x$. 
	Then, 
	\begin{equation}
		|w(z',0)- w_S| \leq C C_1 \dfrac{d^{n+1}}{|S|} d^\alpha
	\end{equation}	
	for some constant $C$ depending only on $n$ and $\alpha$.
\end{lemma}

\begin{proof}
	
	The proof is that of Lemma~C.1 in \cite{CabreQuantitative} (which, in turn, follows that of Chapter~7 in~\cite{GilbargTrudinger}).
	Given $x\in S$, we have
	\begin{equation}
		\begin{split}
			w(x) - w(z',0) &= \int_0^{|x - (z',0)|}  \dfrac{ x - (z',0)}{|x - (z',0)|} \cdot \nabla w \left ( (z',0) + \rho \dfrac{ x - (z',0)}{|x - (z',0)|} \right ) \d \rho \\
			&=  \int_0^{|x - (z',0)|}  w_{r_{z'}} \left ( (z',0) + \rho \dfrac{ x - (z',0)}{|x - (z',0)|} \right ) \d \rho.
		\end{split}		
	\end{equation}
	Averaging in $x\in S$, taking absolute values, and using spherical coordinates centered at $(z',0)$ (i.e., $x = (z',0) + \tilde{\rho} \omega$, with $\tilde{\rho} = |x - (z',0) | = r_{z'}$ and $\omega \in \Sph^{n}_+$, where $\Sph^{n}_+ := \Sph^{n} \cap \{y>0\}$), we get
	\begin{equation}
		\begin{split}
			|w(z',0)  - w_S| &\leq  \left | \dfrac{1}{|S|} \int_S \df x  \int_0^{|x - (z',0)|}  \d \rho \,  w_{r_{z'}} \left ( (z',0) + \rho \dfrac{ x - (z',0)}{|x - (z',0)|} \right ) \right | \\
			&\leq    \dfrac{1}{|S|} \int_{\B^+_d(z')} \df x  \int_0^{d}  \d \rho \left | w_{r_{z'}} \left ( (z',0) + \rho \dfrac{ x - (z',0)}{|x - (z',0)|} \right ) \right | \\
			&= \dfrac{1}{|S|} \int_0^d \df \tilde{\rho} \, \tilde{\rho}^n  \int_{\Sph^n_+} \df \omega  \int_0^{d}  \d \rho \left | w_{r_{z'}} \left ( (z',0) + \rho \omega\right ) \right | \\
			&= \dfrac{d^{n+1}}{(n+1) |S|} \int_{\Sph^n_+} \df \omega  \int_0^{d}  \d \rho \left | w_{r_{z'}} \left ( (z',0) + \rho \omega\right ) \right | \\
			&= \dfrac{d^{n+1}}{(n+1) |S|} \int_{\B^+_d(z')} r_{z'}^{-n}  | w_{r_{z'}} |  \d x.
		\end{split}		
	\end{equation}
	
	Let us now bound this last integral. 
	For this, set
	\begin{equation}
		\varphi (t) := \int_{\B^+_t(z')}  | w_{r_{z'}} |  \d x
	\end{equation}
	and use that
	\begin{equation}
		\varphi' (t) = \int_{\partial \B_t(z') \cap \{y>0\}}  | w_{r_{z'}} |  \d \mathcal{H}^n.
	\end{equation}
	Integrating by parts and using \eqref{Eq:MorreyGrowth} we get
	\begin{equation}
		\begin{split}
			\int_{\B^+_d(z')} r_{z'}^{-n}  | w_{r_{z'}} |  \d x 
			&= \int_0^d \df t \int_{\partial \B_t(z') \cap \{y>0\}} \df \mathcal{H}^n \, t^ {-n} | w_{r_{z'}} |   = \int_0^d t^{-n} \varphi'(t) \d t \\
			& = d^{-n} \varphi (d)  + n \int_0^d t^{-n-1} \varphi(t) \d t  \\
			& \leq C_1 d^{-n} d^{n + \alpha }  + n C_1 \int_0^d t^{-n-1} t^{n + \alpha} \d t  \leq C C_1 d^\alpha,
		\end{split}
	\end{equation}
	establishing the result.
\end{proof}

Once we have the previous result at hand, we can proceed now with the proof of our main theorem, which provides the Hölder estimate for stable solutions in dimensions $n\leq 4$.

\begin{proof}[Proof of \Cref{Th:Holder}]
	The $H^{1/2}$ estimate \eqref{Eq:H1/2} has been already proved in \Cref{Prop:H1ControlledByL1}.
	From now on, we assume that $2 \leq n \leq 4$, and notice that once the Hölder estimate \eqref{Eq:Holder} is established in dimension $n=2$, then by the results of next section (\Cref{Lemma:AddDimensionEquation,Lemma:AddDimensionHsL1s}), it will hold as well for stable solutions $u:\R \to \R$ in dimension $n=1$.
	For this, we simply define $w(x_1,x_2) := u(x_1)$ and apply estimate \eqref{Eq:Holder} to $w$.

	Now, since $2 \leq n \leq 4$, by \Cref{Prop:GeometricDecayOfDirichletInt} we know that 
	$$
	\int_{\B_{R}^+}  |v_r| \d x  \leq \left( \int_{\B_{R}^+} r^{n-1} \d x\right)^{1/2} 
	\left( \int_{\B_{R}^+} r^{1-n} |v_r|^2 \d x\right)^{1/2} 
	\leq C  R^{n + \alpha} \norm{\nabla v}_{L^2(\B^+_{3/4})}
	$$
	for $R \leq 1/2$, where $\alpha$ and $C$ are some positive dimensional constants.
	Replacing the origin by any $z'\in \overline{B}_{1/8} \subset \R^n$ in the above inequality (since the equation solved by $u$ is invariant under translations in $\R^n$), we get that
	\begin{equation}
		\int_{\B_{R}^+(z')}  |v_{r_{z'}} | \d x  \leq C  R^{n + \alpha} \norm{\nabla v}_{L^2(\B^+_{3/4}(z'))} \leq C  R^{n + \alpha} \norm{\nabla v}_{L^2(\B^+_{7/8})} \quad \text{ for all } R \leq 1/2.
	\end{equation}
	Using \Cref{Prop:H1ControlledByL1} to control $\norm{\nabla v}_{L^2(\B^+_{7/8})}$ (after a covering and scaling argument) by $\norm{u}_{L^1_{1/2}(\R^n)}$, we conclude
	\begin{equation}
		\label{Eq:GrowthRadialDerivatives}
		\int_{\B_{R}^+(z')}  |v_{r_{z'}} | \d x  \leq C_0  R^{n + \alpha} \norm{u}_{L^1_{1/2}(\R^n)} \quad \text{ for all } z'\in \overline{B}_{1/8} \text{ and } R \leq 1/2,
	\end{equation}
	for some other dimensional constant $C_0$.

	From these geometric-decay bounds for radial derivatives, we now obtain a Hölder estimate for $u$ in $\overline{B}_{1/8}$.
	For this, given $z'$ and $\tilde{z}' $ in $\overline{B}_{1/8}$, we set
	\begin{equation}
		d := |z' - \tilde{z}'| \leq 1/4 \quad \text{ and } \quad  S:= \B^+_d (z') \cap \B^+_d (\tilde{z}').
	\end{equation}
	Note that  $|S|= c(n) d^{n+1}$ for some dimensional constant $c(n)$.
	Hence, using \Cref{Lemma:MorreyAverage} in $\B^+_d (z')$ and $\B^+_d (\tilde{z}')$ ---in both cases with $w=v$, $C_1 = C_0 \norm{u}_{L^1_{1/2}(\R^n)} $, and with $d$ and $S$ as defined above---, we obtain
	\begin{equation}
		|u(z') - u(\tilde{z}') | \leq |v(z',0) - v_S | + |v_S - v(\tilde{z}',0) | \leq C C_1 d^\alpha = C  C_0 \norm{u}_{L^1_{1/2}(\R^n)} |z' - \tilde{z}'|^\alpha
	\end{equation}
	for some constant $C$ depending only on $n$ and $\alpha$.
	Hence, 
	$$
	\seminorm{u}_{C^\alpha(\overline{B}_{ 1/8}) } \leq C \norm{u}_{L^1_{1/2}(\R^n)}.
	$$
	
	Now, we bound the $L^\infty$ norm of $u$ as follows. 
	For $x'$ and $z'$ in $B_{1/8}$, by the previous estimate we have
	$$
	|u(x')| \leq |u(x')-u(z')| + |u(z')| \leq C \norm{u}_{L^1_{1/2}(\R^n)} + |u(z')|.
	$$
	Integrating  with respect to $z'$ in $B_{1/8}$ we deduce
	$$
	\norm{u}_{L^\infty(B_{1/8})} \leq C \left( \norm{u}_{L^1_{1/2}(\R^n)} + \norm{u}_{L^1(B_{1/8})} \right ) \leq C \norm{u}_{L^1_{1/2}(\R^n)}.
	$$
	We conclude that
	$$
	\norm{u}_{C^\alpha(\overline{B}_{ 1/8})} \leq C \norm{u}_{L^1_{1/2}(\R^n)}.
	$$
	
	Finally, we prove the claimed estimate in $\overline{B}_{1/2}$ by using a standard covering and scaling argument.	
\end{proof}

Once the main theorem is proved, we can establish  \Cref{Coro:LinftyConvexDiriclet}.

\begin{proof}[Proof of \Cref{Coro:LinftyConvexDiriclet}]
	We may assume $u\not \equiv 0$, and thus $u> 0$ in $\Omega$ by the maximum principle.
	Now, on the one hand, since $\Omega$ is convex, by Proposition~1.8 of~\cite{RosOtonSerra-Extremal}, there exist two positive constants $\delta$ and $C_\Omega$, both depending only on $\Omega$, such that
	\begin{equation}
		\label{Eq:DeltaBoundaryEstimateConvex}
		\norm{u}_{L^\infty (\Omega\setminus K_{\delta})} \leq C_\Omega \norm{u}_{L^1(\Omega)},
	\end{equation}
	where $	K_\delta:=\{x'\in\Omega : \dist(x',\partial\Omega)\geq\delta\}$.
	
	On the other hand, by regularity theory (see \Cref{Remark:Regularity}), $u\in C^2(\Omega)$.
	Thus, we can use \Cref{Th:Holder}, that together with a covering and scaling argument leads to
	\begin{equation}
		\label{Eq:InteriorEstimateConvex}
		\norm{u}_{L^\infty (K_{\delta})} \leq C_\Omega \norm{u}_{L^1(\Omega)}  ,
	\end{equation}
	for some constant $C_\Omega$ depending only on $\Omega$ (to control the $L^1_{1/2}$ norm of $u$, recall that $u \equiv 0$ in $\R^n \setminus \Omega$).
	
	Combining both estimates we conclude the desired result.
\end{proof}

With \Cref{Coro:LinftyConvexDiriclet} at hand, we can establish our boundedness result for the extremal solution.

\begin{proof}[Proof of \Cref{Coro:Extremal}]
	
	For $\lambda < \lambda^\star$, let $u_\lambda$ be the minimal solution to \eqref{Eq:ExtremalProblem} with $\s=1/2$. 
	Since $u_\lambda \in L^\infty(\Omega) \cap H^{1/2}(\R^n)$ and it is a stable solution, from \Cref{Coro:LinftyConvexDiriclet} we obtain
	$$
	\norm{u_\lambda}_{L^\infty (\Omega)} \leq C_\Omega \norm{u_\lambda}_{L^1(\Omega)}.
	$$
	But now, since  $u^\star \in L^1(\Omega)$ and $0 \leq u_\lambda\leq u^\star$, the right-hand side of the previous inequality is bounded independently of $\lambda$.
	As a consequence, we deduce our result by letting $\lambda \to \lambda^\star$.
\end{proof}

We conclude the section by establishing our Liouville result. 

\begin{proof}[Proof of \Cref{Coro:Liouville}]
	First, note that if $\tilde{u}$ is a stable solution to $\halflaplacian \tilde{u} = f(\tilde{u})$ in $B_1$, and $\tilde{v}$ is its harmonic extension in $\R^{n+1}_+$, then 
	\begin{equation}
		\label{Eq:HolderEstimateL1Ext}
		[\tilde{u}]_{C^\alpha (\overline{B}_{1/2})} \leq C \norm{\tilde{v}}_{L^1(\B_{1}^+)} \quad \text{ if } 1 \leq n \leq 4,
	\end{equation}
	where $\alpha > 0$ and $C$ are dimensional constants.
	To see this, it is enough to carry out the  proof of \Cref{Th:Holder}, but using estimate \eqref{Eq:H1ControledByL1Ext} instead of \eqref{Eq:H^1Estimate} to control the $H^1$ seminorm of $\tilde{v}$ in \eqref{Eq:GrowthRadialDerivatives}.

	Our proof now follows that of  \cite{DupaigneFarina}. 
	Given $u$ as in the statement of the corollary, for each $R>2$ define $u_R(x') := u(Rx')$.
	Since the harmonic extension of $u_R$ in $\R^{n+1}_+$ is $v_R(x) := v(Rx)$, by applying the estimate \eqref{Eq:HolderEstimateL1Ext} to these rescaled functions (which are stable solutions to $\halflaplacian u_R = Rf(u_R)$ in $B_1$) we obtain
	\begin{equation}
		|u(x') - u(z') |\leq C R^{-\alpha} |x' - z'|^\alpha \fint_{\B_R^+} |v| \d x \quad \text{ for } x'  \text{ and }  z'  \text{ in } B_{R/2}.
	\end{equation}
	Note that here we use crucially that the estimate \eqref{Eq:HolderEstimateL1Ext}  does not depend on the nonlinearity.

	Now, we need to bound the average of $|v|$ in $\B_R^+$. 
	To do it, we claim first that there exists a positive constant $c_1$ for which
	\begin{equation}
		\label{Eq:LowerBoundvExt}
		v(x) \geq - c_0 ( \log |x| + c_1) \quad \text{ if } |x| \geq 1,
	\end{equation} 
	where $c_0$ is the constant appearing in the lower bound \eqref{Eq:LowerBoundu} for $u$.
	Assuming this claim to be true, we obtain
	\begin{equation}
		\begin{split}
			|v(x) | &\leq |v(x)  +  c_0 ( \log |x| + c_1)| + | c_0 ( \log |x| + c_1)| = v(x) + 2  c_0 ( \log |x| + c_1).
		\end{split}	
	\end{equation}
	Moreover, since $v$ is harmonic and has nonnegative flux on $B_1$, the mean value property
	\begin{equation}
		\fint_{\B_R^+} v \d x \leq v(0)
	\end{equation}
	holds.\footnote{This follows from the same usual proof of the mean value property for superharmonic functions in full balls, by adapting it to half-balls. Or, alternatively, by applying it to the even reflexion of $v$ across $\{y=0\}$ (which is superharmonic in the weak $H^1$ sense).
	}
	
	As a consequence, using the previous two bounds we get
	\begin{equation}
		\fint_{\B_R^+} |v| \d x \leq \fint_{\B_R^+} v \d x + 2 c_0 \left(\log R + c_1 \right) \leq v(0) + 2 c_0 \left(\log R + c_1 \right)
	\end{equation}
	and thus
	\begin{equation}
		|u(x') - u(z') |\leq C R^{-\alpha} |x' - z'|^\alpha \big(v(0) + 2 c_0 \left(\log R + c_1 \right) \big).
	\end{equation}	
	Letting $R\to +\infty$ we deduce that $u$ is constant.
	
	It remains to establish \eqref{Eq:LowerBoundvExt}.
	We use the Poisson kernel \eqref{Eq:vPoisson} to express $v$ in terms of $u$, together with the lower bound \eqref{Eq:LowerBoundu}. For $x=(x',y)\in\R^{n+1}_+$ with $|x|\geq 1$, we see that
	\begin{equation}
		\begin{split}
			v(x) &\geq -c_0 \, p_{n,1/2} \int_{\R^n} \dfrac{y \log (2 + |z'|)}{(|x'-z'|^2 + y^2)^{\frac{n+1}{2}}} \d z'
			\\	&\geq  -c_0  \, p_{n,1/2} \int_{\R^n} \dfrac{y \log (|x|(2 +  |z'|/|x|))}{(|x'-z'|^2 + y^2)^{\frac{n+1}{2}}} \d z' \\
			&=  - c_0 \left ( \log|x| +  p_{n,1/2} \int_{\R^n} \dfrac{y \log (2 +  |z'|/|x|)}{(|x'-z'|^2 + y^2)^{\frac{n+1}{2}}} \d z' \right ) \\
			&=  - c_0 \left ( \log|x| +   p_{n,1/2}\, \dfrac{y}{|x|}  \int_{\R^n} \dfrac{ \log (2 +  |\tilde{z}'|)}{(|x'/|x|-\tilde{z}'|^2 + (y/|x|)^2)^{\frac{n+1}{2}}} \d \tilde{z}' \right ) \\
			&=: - c_0 \left ( \log|x| +  p_{n,1/2}\,\varphi (x/|x|) \right).
		\end{split}
	\end{equation}
	
	Finally, let us show that the function $\varphi:\Sph^n_+\to\R$ defined above is bounded.
	To do this, note first that $\varphi=\varphi(x/|x|)$ does not depend on $x'/|x|$ and can be written as 
	\begin{equation}
		\tilde{y} \int_{\R^n} \dfrac{ \log (2 +  |\tilde{z}'|)}{(|e_1-\tilde{z}'|^2 + \tilde{y}^2)^{\frac{n+1}{2}}} \d \tilde{z}' ,
	\end{equation}
	where $\tilde{y}:=y/|x|\in (0,1]$. Since $|\tilde{y}|\leq 1$, it suffices to control this quantity when integrating only in $\{|e_1-\tilde{z}'| <1\}$. But we have 	
	\begin{equation}
		%		\begin{split}
		\tilde{y} \int_{B_1(e_1)} \dfrac{ \log (2 +  |\tilde{z}'|)}{(|e_1-\tilde{z}'|^2 + \tilde{y}^2)^{\frac{n+1}{2}}} \d \tilde{z}' \leq C	\int_0^1  \dfrac{\tilde{y} \rho^{n-1}}{(\rho^2 + \tilde{y}^2)^{\frac{n+1}{2}} }  \d \rho   \leq 
		C 	\int_0^{+\infty}   \dfrac{ t^{n-1}}{(t^2 + 1)^{\frac{n+1}{2}} }  \d t < +\infty.
		%		\end{split}
	\end{equation} 
	This establishes the claim \eqref{Eq:LowerBoundvExt} and concludes the proof.	
\end{proof}

%%%%%%%%%%%%%%%%%%%%%%%%%%%%%%%%%
\section{Adding artificial variables}
%%%%%%%%%%%%%%%%%%%%%%%%%%%%%%%%%
\label{Sec:AddDimensions}

In this section we present the two results that allowed us to establish the Hölder estimate of \Cref{Th:Holder} in dimension $n=1$ from the estimate in higher dimensions, by adding artificial variables.
In this section, points in $\R^m = \R^{m-1} \times \R$ will be denoted by $x=(x',x_m)$.

The first result concerns the control of the $L^1_\s$ norm when adding artificial variables.

\begin{lemma}
	\label{Lemma:AddDimensionHsL1s}
	
	Let $m\geq 2$ and let $u: \R^{m-1} \to \R$ belong to $ L^1_\s(\R^{m-1})$.
	Define $w: \R^m \to \R$ by
	$$
	w(x',x_m) := u(x').
	$$
	
	Then,
	$$
	\norm{w}_{L^1_\s(\R^m)} \leq C \norm{u}_{L^1_\s(\R^{m-1})}
	$$
	for some constant $C$ depending only on $m$ and $\s$.
\end{lemma}

\begin{proof}	
	Let $Q^m = (-1,1)^m \subset \R^m$.
	We have
	$$
	\norm{w}_{L^1_\s(\R^m)} = \int_{Q^m} \dfrac{|u(x')|}{(1 + |x'|^2 + x_m^2)^{\frac{m + 2\s}{2}}} \d x +  \int_{\R^m \setminus Q^m} \dfrac{|u(x')|}{(1 + |x'|^2 + x_m^2)^{\frac{m + 2\s}{2}}} \d x.
	$$
	The first integral is bounded easily using that $Q^m = Q^{m-1} \times (-1,1)$:
	$$
	\int_{Q^m} \dfrac{|u(x')|}{(1 + |x'|^2 + x_m^2)^{\frac{m + 2\s}{2}}} \d x \leq 2 \int_{Q^{m-1}} |u(x')|\d x' \leq C \norm{u}_{L^1_\s(\R^{m-1})}.
	$$
	To estimate the second one, we split it further:
	$$
	\int_{\R^m \setminus Q^m} \dfrac{|u(x')|}{(1 + |x'|^2 + x_m^2)^{\frac{m + 2\s}{2}}} \d x = I_1 + I_2,
	$$ 
	where
	$$
	I_1 := \int_{-1}^1 \df x_m \int_{\R^{m-1} \setminus Q^{m-1}} \df x' \ \dfrac{|u(x')|}{(1 + |x'|^2 + x_m^2)^{\frac{m + 2\s}{2}}}   
	$$
	and 
	$$
	I_2 := 2 \int_1^{+\infty} \df x_m \int_{\R^{m-1}} \df x' \ \dfrac{|u(x')|}{(1 + |x'|^2 + x_m^2)^{\frac{m + 2\s}{2}}}   .
	$$
	
	Now, on the one hand
	\begin{align}
		I_1 & \leq 2 \int_{\R^{m-1} \setminus Q^{m-1}}  \dfrac{|u(x')|}{(1 + |x'|^2 )^{\frac{m - 1 + 2\s}{2}}}   \d x' \leq C \norm{u}_{L^1_\s(\R^{m-1})}.
	\end{align}
	On the other hand, 
	\begin{align}
		I_2 & =  2 \int_{\R^{m-1}} \df x' \ |u(x')|  \int_1^{+\infty}  \dfrac{\df x_m }{(1 + |x'|^2 + x_m^2)^{\frac{m + 2\s}{2}}},
	\end{align}
	and note that, for $\lambda >0$,
	\begin{align}
		\int_1^{+\infty}  \dfrac{\df x_m }{(\lambda^2 + x_m^2)^{\frac{m + 2\s}{2}}}  &= \dfrac{1}{\lambda^{m + 2\s}} \int_1^{+\infty}  \dfrac{\df x_m }{(1 + (x_m/\lambda) ^2)^{\frac{m + 2\s}{2}}} \\
		& \leq \dfrac{1}{\lambda^{m - 1 + 2\s}} \int_{0}^{+\infty}  \dfrac{\df t }{(1 + t ^2)^{\frac{m + 2\s}{2}}} \leq \dfrac{C}{\lambda^{m - 1 + 2\s}}.
	\end{align}
	This yields $I_2 \leq C \norm{u}_{L^1_\s(\R^{m-1})}$ (taking $\lambda = \sqrt{1 + |x'|^2}$), which concludes the proof.	
\end{proof}

Now we prove that  stability is preserved after the addition of artificial variables.
Here we denote by $B_1^m$ the unit ball in $\R^m$.

\begin{lemma}
	\label{Lemma:AddDimensionEquation}
	Let $m\geq 2$ and let $u: \R^{m-1} \to \R$, with $u\in C^2(B_1^{m-1})\cap L^1_\s(\R^{m-1})$, be a  stable solution of $\fraclaplacian u = f(u)$ in $B_1^{m-1}$. 
	Define $w: \R^m \to \R$ by
	$$
	w(x',x_m) := u(x').
	$$
	
	Then, $w\in C^2(B_1^{m})\cap L^1_\s(\R^{m})$ and it is a stable solution of $\fraclaplacian w = f(w)$ in $B_1^m$.
\end{lemma}

\begin{proof}
	First, that $w\in L^1_\s(\R^{m})$ follows from \Cref{Lemma:AddDimensionHsL1s}, and a straightforward computation shows that $w$ solves $\fraclaplacian w = f(w)$ in $B_1^m$ (see for instance \cite[Lemma~2.1]{RosOtonSerra-GeneralBoundaryRegularity}).	
	Now, let us check that $w$ is stable. 
	We will show that the stability inequality \eqref{Eq:StabilityExtension} holds for all $C^\infty$ functions $\xi = \xi (x,y)$ with compact support in $\R^{m+1}_+ \cup B_1^{m}$. 
	Then, the result will follow by density.

	Given such a function $\xi$, we define
	$$
	\overline{\xi}^2(x',y) := \int_{\R} \xi^2(x', x_m, y) \d x_m .
	$$
	We have
	\begin{align}
		\int_{ B_1^m } f'(w(x)) \xi^2(x, 0) \d x &= \int_{ B_1^{m-1} } \df x' f'(u(x')) \int_{\R} \df x_m \, \xi^2(x', x_m, 0) \\
		&=  \int_{ B_1^{m-1} } f'(u(x')) \overline{\xi}^2(x',0) \d x',
	\end{align}
	and using the stability of $u$ in $B_1^{m-1}$, we deduce
	\begin{equation}
		\label{Eq:DimAddStability}
		\int_{ B_1^m } f'(w) \xi^2 \d x \leq d_\s \int_0^{+\infty} \df y\ y^{1 - 2\s} \int_{\R^{m-1}} \df x' \ |\nabla_{(x'\!,y)} \overline{\xi} (x', y) |^2.
	\end{equation}
	
	Now, for $(x',y) \in \{ \overline \xi \neq 0\}$ and $i = 1, \ldots, m-1$ we have
	$$
	\overline{\xi}_{x_i}(x',y) = \overline{\xi} (x',y)^{-1} \int_{\R} \xi (x', x_m,y) \xi_{x_i} (x', x_m,y)\d x_m,
	$$
	and the same holds for the derivative with respect to the extension variable $y$.
	Using the Cauchy-Schwarz inequality we see that
	$$
	\overline{\xi}_{x_i}^2 (x',y) \leq  \int_{\R} \xi_{x_i}^2(x', x_m,y)\d x_m,
	$$
	and similarly for $\overline{\xi}_{y}^2$.
	Hence
	$$
	|\nabla_{(x'\!,y)} \overline{\xi} (x', y) |^2 \leq \int_{\R} |\nabla \xi (x', x_m,  y) |^2 \d x_m .
	$$
	Using this in \eqref{Eq:DimAddStability}, and the fact that since $\overline \xi$ is Lipschitz, $\nabla \overline \xi = 0$ a.e. in $\{ \overline \xi = 0\}$ (see for instance~\cite[Theorem~6.19]{LiebLoss}), we obtain that
	$$
	\int_{ B_1^m } f'(w) \xi^2 \d x \leq d_\s \int_0^{+\infty} \df y \ y^{1 - 2\s}  \int_{\R^m} \df x \  |\nabla \xi (x,  y) |^2.
	$$
	This establishes the stability of $w$.	
\end{proof}

\appendix
\gdef\thesection{\Alph{section}} % corrected redefinition of "\thesection"
\makeatletter
\renewcommand\@seccntformat[1]{Appendix \csname the#1\endcsname.\hspace{0.5em}}
\makeatother

%%%%%%%%%%%%%%%%%%%%%%%%%%%%%%%%%%%%%%%%%%%%%%%%%%%%%%%%%%%%%
%%%%%%%%%%%%%%%%%%%%%%%%%%%%%%%%%%%%%%%%%%%%%%%%%%%%%%%%%%%%%
\section{Auxiliary regularity lemmata}
\label{Sec:PreliminarEstimates}
%%%%%%%%%%%%%%%%%%%%%%%%%%%%%%%%%%%%%%%%%%%%%%%%%%%%%%%%%%%%%
%%%%%%%%%%%%%%%%%%%%%%%%%%%%%%%%%%%%%%%%%%%%%%%%%%%%%%%%%%%%%

In this section we collect some auxiliary results used along the paper.
Recall that $a := 1-2\s$ and that the $\s$-harmonic extension of a function $u:\R^n \to \R$ is iven by 
\begin{equation}
	\label{Eq:vPoisson}
	v(x',y) = \int_{\R^n} u(z') P_\s(x'-z',y) \d z', \quad \text{ with } \quad 
	P_\s (\bar{x}', y) = p_{n,\s} \dfrac{y^{2\s} }{(|\bar{x}'|^2 + y^2)^\frac{n + 2\s}{2}}.
\end{equation}
Here, $p_{n,\s}$ is a normalizing positive constant, depending only on $n$ and $\s$, which makes $P_\s(\cdot, y)$ integrate~$1$ in $\R^n$ for every $y > 0$.

We first present a result stating that the weighted $L^p$ norm of the $\s$-harmonic extension of a function can be controlled in terms of $L^p$ and $L^1_\s$ norms of its trace.
It is a simple application of Young's convolution inequality.

\begin{lemma}
	\label{Lemma:LpExtension}
	Let $n\geq 1$, $\s \in (0,1)$, $p\in [1,+\infty)$, $R>0$, and $u \in L^p (B_{2R}) \cap L^1_\s (\R^n)$.
	Let $a = 1-2\s$ and let $v$ be the $\s$-harmonic extension of $u$.
	
	Then, 
	$$
	\norm{v}_{L^p (\B_R^+, y^a)} \leq C_R \left(  \norm{u}_{L^p (B_{2R})} + \norm{u}_{L^1_\s (\R^n)} \right)
	$$
	for some constant $C_R$ depending only on $n$, $\s$, $p$, and $R$.
\end{lemma}

\begin{proof}
	First, note that
	\begin{align}
		\int_{\B_{R}^+} y^a |v|^p \d x 
		&\leq  \int_0^R \df y \, y^a\int_{B_{R}} \df x' \left| \int_{\R^n} u(z')  P_\s(x'-z',y)\d z' \right|^p \\
		& \leq 2^p \int_0^R \df y \, y^a \int_{B_{R}} \df x' \left| \int_{B_{2R}} u(z')  P_\s(x'-z',y)\d z' \right|^p \\
		& \quad \quad + 2^p \int_0^R\df y  \, y^a  \int_{B_{R}} \df x' \left| \int_{\R^n\setminus B_{2R}} u(z')  P_\s(x'-z',y)\d z' \right|^p.
	\end{align}
	Now, on the one hand, defining $\widetilde{u} := u \chi_{B_{2R}}$, for $y>0$ we have
	\begin{align}
		\int_{B_{R}} \df x' \left| \int_{B_{2R}} u(z')  P_\s(x'-z',y)\d z' \right|^p 
		&= \int_{B_{R}}  |\widetilde{u} * P_\s(\cdot, y) |^p (x') \d x' \\
		& \leq \norm{\widetilde{u} * P_\s(\cdot, y) }_{L^p(\R^n)}^p \\
		& \leq \norm{\widetilde{u} }_{L^p(\R^n)}^p \norm{ P_\s(\cdot, y) }_{L^1(\R^n)}^p  =  \norm{ u }_{L^p(B_{2R})}^p,
	\end{align}
	where we have used Young's convolution inequality.
	On the other hand, we claim that 
	\begin{equation}
		\label{Eq:PoissonBound}
		P_\s(x'-z',y) \leq \overline{C}_R \dfrac{y^{2\s}}{(1 + |z'|^2)^{\frac{n+ 2\s}{2}}} \quad \text{if } x'\in B_{R} \text{ and } z' \in \R^n \setminus B_{2R},
	\end{equation}
	where the constant $\overline{C}_R$ depends only on $n$, $\s$, and $R$. 
	Indeed, since $|x'| \leq |z'|/2$, we have that
	\begin{equation}
		1 + |z'|^2 \leq  \dfrac{1 + 4R^2}{R^2} \dfrac{|z'|^2}{4} \leq  \dfrac{1 + 4R^2}{R^2} |x'-z'|^2  \leq  \dfrac{1 + 4R^2}{R^2} (|x'-z'|^2 + y^2).
	\end{equation}
	Therefore
	$$
	\int_{B_{R}} \df x' \left| \int_{\R^n\setminus B_{2R}} u(z')  P_\s(x'-z',y)\d z' \right|^p \leq C_R y^{2\s p}   \norm{u}_{L^1_\s(\R^n)}^p, 
	$$
	with $C_R$ depending only on $n$, $\s$, $p$, and $R$. 
	
	The desired estimate follows from the fact that $y^a$ is integrable in $(0,R)$.
\end{proof}

We next establish an $L^\infty$ estimate for the $\s$-harmonic extension of a function, as well as for its first and second horizontal derivatives (i.e., its derivatives in $x'\in \R^n$).
The result will be used to prove \Cref{Lemma:RegularityHorizontalGrad} below.

\begin{lemma}
	\label{Lemma:EstimatesHalfBallsExtension}
	Let $n\geq 1$, $\s\in (0,1)$, $R>0$, and $u \in C^2(\overline{B}_{4R}) \cap L^1_\s(\R^n)$.
	Let $v$ be the $\s$-harmonic extension of $u$. 
	
	Then, 
	\begin{equation}
		\norm{v}_{L^\infty({\B}_{R}^+)} + \norm{\nabla_{x'}v}_{L^\infty({\B}_{R}^+)} + \norm{D^2_{x'} v}_{L^\infty({\B}_{R}^+)} \leq C_R \bpar{\norm{u}_{C^2(\overline{B}_{4R})} + \norm{u}_{L^{1}_\s(\R^n)}}
	\end{equation}
	for a constant $C_R$ depending only on $n$, $\s$, and $R$.	
\end{lemma}

\begin{proof}	
	Let $\zeta: \R^n \to [0,1]$ be a smooth cut-off function such that $\zeta \equiv 1$ in $B_{2R}$ and $\zeta \equiv 0$ in $\R^n \setminus B_{4R}$.
	Then, writing $u = \zeta u + (1-\zeta) u$, for $(x',y) \in \B^+_R$ we have
	\begin{equation}
		\label{Eq:vPoissonSplitCutoff}
		v(x',y) = \int_{\R^n} \zeta (x'-z') u(x'-z') P_\s(z',y)  \d z' + \int_{\R^n} (1-\zeta(z'))u(z') P_\s(x'-z',y) \d z'.
	\end{equation}
	Now, the first term is estimated by $\norm{u}_{L^\infty(B_{4R})}$ using simply that $\zeta$ has compact support in $B_{4R}$ and that $P_\s$ is positive and integrates $1$, while for the second one we note that $1 - \zeta$ is zero in $B_{2R}$ and thus we can use \eqref{Eq:PoissonBound} to estimate the integral by $C_R \norm{u}_{L^{1}_\s(\R^n)}$.
	
	Last, to estimate the horizontal derivatives of $v$ we proceed similarly, now differentiating \eqref{Eq:vPoissonSplitCutoff} and using, for the second integral (where $|x'-z'| \geq |z'|-|x'| \geq R$), that, for $i,j = 1, \ldots,n$, 
	\begin{equation}
		| \partial_{i}P_\s (x'- z', y)| + | \partial_{i} \partial_{j} P_\s (x'- z', y)| \leq C_R P_\s (x'- z',y) \quad \text{ if } |x'- z'|\geq R
	\end{equation}
	for some constant $C_R$ depending only on $n$, $\s$, and $R$.
	This last bound can be easily obtained from the explicit expression of the Poisson kernel $P_s$ just noticing that $|x'- z'|\geq R$ yields
	\begin{equation}
		\dfrac{|x_i - z_i|}{|x'-z'|^2 + y^2} + \dfrac{1}{|x'-z'|^2 + y^2} \leq \dfrac{1}{|x'-z'|} + \dfrac{1}{|x'-z'|^2} \leq \dfrac{1}{R} + \dfrac{1}{R^2}.
	\end{equation}
\end{proof}

In the following lemma we collect some regularity results for the $\s$-harmonic extension of a solution to a fractional semilinear equation in $B_1$.
These are the ingredients that we have referred to in our computations through the paper.
The main issue here is that $u\in L^1_\s(\R^n)$ is the only control that we have of $u$ outside $B_1$.
Instead, under more restrictive assumptions on the solution ---for instance in \cite{CabreSireI} for entire solutions in $L^\infty(\R^n)$---, most of these results are known.

\begin{lemma}
	\label{Lemma:RegularityHorizontalGrad}
	Let $n\geq 1$, $\s \in (0,1)$, and let $u\in C^2(B_1)\cap L^1_\s(\R^n)$ be a solution to $\fraclaplacian u = f(u)$ in $B_1\subset \R^n$, where $f$ is a  $C^{1,\gamma}$ function for some $\gamma >0$.	
	Let $a = 1-2\s$ and let $v$ be the $\s$-harmonic extension of $u$.
	
	Then, for every $R<1$,
	\begin{enumerate}[label=(\alph*)]
		\item $v \in H^1(\B^+_R, y^a)$.
		
		\item The functions $v$, $y^a \partial_y v$, $\nabla_{x'}v$, and $y^a \partial_y \nabla_{x'}v$ are continuous in $\overline{\B^+_R}$.
		Moreover, 
		\begin{equation}
			\label{Eq:NeumanGradxv}
			- d_\s \lim_{y \downarrow 0} y^a \partial_y \nabla_{x'}v = f'(u) \nabla_{x'}u \quad \text{ in } B_R,
		\end{equation}
		where $d_\s$ is the constant appearing in \eqref{Eq:DirichletToNeumannRelation}.
		
		\item $D^2_{x'} v$ is bounded in $\B^+_R$. 
		For $\s = 1/2$ we additionally\footnote{Note that, for $\s< 1/2$, in general $v$ is not even $C^1$ up to $\{y=0\}$. Indeed, since $y^a v_y$ is continuous up to the boundary, $v_y$ will generally blow up at $\{y=0\}$.} have that $v\in C^2(\overline{\B^+_R})$.
	\end{enumerate}
	
\end{lemma}

\begin{proof}
	Since $u\in L^1_\s(\R^n)$, the function $v$ given by \eqref{Eq:vPoisson} is well-defined and smooth in $\R^{n+1}_+$, and	satisfies $\div (y^a \nabla v) = 0$ in $\R^{n+1}_+$.
	Therefore, the main issue is to prove that the regularity results stated in the lemma hold up to $\{y=0\}$.
	
	First, to show that $v$ is continuous up to $\{y=0\}$ we simply use its expression as a convolution with $P_s$. 
	Indeed, from \eqref{Eq:vPoissonSplitCutoff} we see that the first integral in that expression is the convolution of an approximation of the identity with a compactly supported function which is continuous in $B_1$ (and thus the result is continuous in $B_1$ as $y \downarrow 0$), while the second integral can be bounded by $C_R y^{2\s} \norm{u}_{L^1_\s(\R^n)}$ thanks to \eqref{Eq:PoissonBound}.

	Next, it is well known and easy to verify ---using \eqref{Eq:vPoisson}--- that $y^a v_y$ is continuous in $\overline{\B^+_R}$, and that 
	\begin{equation}
		\label{Eq:NeumanAp}
		d_\s \dfrac{\partial v}{\partial \nu^a} := - d_\s \lim_{y\downarrow 0} y^a v_y = \fraclaplacian u = f(u) \quad \text{ in } B_R.
	\end{equation}
	
	That $v \in H^1(\B^+_R, y^a)$ follows readily using that $y^a$ and $y^{-a}$ are integrable.
	Indeed, on the one hand $v$ and its horizontal derivatives can be controlled in $L^2(\B^+_R, y^a)$ using  \Cref{Lemma:EstimatesHalfBallsExtension} (appropriately rescaled and after a covering argument). 
	On the other hand, to control the vertical derivative (in $y$) we use that $y^a v_y$ is continuous up to $\{y=0\}$, and thus $y^a v_y^2 \leq C y^{-a}$, from which the desired integrability follows.
	
	Let us now establish the regularity of $v_{x_i}$ for $i= 1,\ldots,n$. 
	Note that we cannot use the previous results applied to $v_{x_i}$ instead of $v$, since we do not have any control of $v_{x_i}$ outside~$B_1$. 
	To carry out the proof, we consider the weak equation solved by $v$, i.e., 
	\begin{equation}
		\label{Eq:vWeakSol}
		d_\s \int_{\B^+_1} y^a \nabla v \cdot \nabla \varphi \d x = \int_{B_1}  f(u) \varphi \d x'
	\end{equation}
	for every $\varphi \in C^1(\overline{\B^+_1})$ with compact support in $\B^+_1\cup B_1$.
	Note that \eqref{Eq:vWeakSol} follows readily from the equation of $v$ after integrating by parts in $\B_{1}^+\cap \{y>\delta\}$ (where $v$ is smooth) and letting $\delta \to 0$, using \eqref{Eq:NeumanAp}.
	Now, for $h>0$, we consider the difference quotient
	\begin{equation}
		D_h^i v (x) := \dfrac{v(x+ h e_i) - v(x)}{h}.
	\end{equation} 
	From \eqref{Eq:vWeakSol}, we get that for every $x_0'\in B_R$ and $\rho>0$ such that $B_{2\rho}(x_0') \subset B_1$, the function $D_h^i v $	(with $h<\rho$) solves weakly
	\begin{equation}
		\beqc{\PDEsystem}
		\div( y^a \nabla D_h^i v) & = & 0& \text{ in } \B^+_\rho(x_0'),\\
		d_\s \dfrac{\partial D_h^i v }{\partial \nu^a} & = & D_h^i [f(v)] & \text{ in } B_\rho(x_0').
		\eeqc
	\end{equation}
	
	Now, using Lemma~4.5 of \cite{CabreSireI}, it follows that
	\begin{equation}
		\label{Eq:DifQuotientsHolder}
		\norm{D_h^i v}_{C^\alpha (\overline{\B^+_{\rho/2}}(x_0'))} + \norm{y^a \partial_y D_h^i v}_{C^\alpha (\overline{\B^+_{\rho/2}}(x_0'))} \leq C_h
	\end{equation}
	for some $\alpha$ depending on $n$, $\s$, and $\gamma$, and a constant $C_h$ depending on $n$, $\s$, $\rho$, $\norm{D_h^i v}_{L^\infty (\B^+_{\rho}(x_0'))}$, and $\norm{D_h^i [f(u)]}_{C^\gamma (\overline{B}_{\rho}(x_0'))}$.
	To establish the regularity for $v_{x_i}$, we need to show that $C_h$ can be controlled uniformly as $h \to 0$.
	On the one hand, $\norm{D_h^i v}_{L^\infty (\B^+_{\rho}(x_0'))}$ is controlled by  $\norm{v_{x_i}}_{L^\infty (\B^+_{\rho}(x_0'))}$, which is finite thanks to \Cref{Lemma:EstimatesHalfBallsExtension} (applied after appropriate translations and rescalings).	
	On the other hand, $\norm{D_h^i [f(u)]}_{C^\gamma (\overline{B}_{\rho}(x_0'))}$ can be controlled by the norm $\norm{f'(u) u_{x_i}}_{C^\gamma (\overline{B}_{\rho}(x_0'))}$, which is also bounded by hypothesis.
	Thus, taking the limit $h\to 0$ in \eqref{Eq:DifQuotientsHolder}, we establish the Hölder regularity for $v_{x_i}$ and $y^a \partial_y v_{x_i}$ in $\overline{\B^+_{\rho/2}}(x_0')$.
	Since $x_0'$ can be taken arbitrarily in $B_R$, we get the desired result.
	Note that once we have proved that $y^a \partial_y v_{x_i}$ is continuous up to $\{y=0\}$, the identity \eqref{Eq:NeumanGradxv} follows readily.
	
	Finally,  \Cref{Lemma:EstimatesHalfBallsExtension} established the boundedness of $D^2_{x'}v$.
	To conclude the proof, let us show that $v\in C^2(\overline{\B^+_R})$ when  $\s=1/2$.
	For this, take $\zeta:\R^n \to \R$ a cut-off function with compact support in $B_{R+ 2\rho}$ for some $\rho<(1- R)/2$, and such that $\zeta \equiv 1$ in $B_{R+ \rho}$.
	Define $w(x,y):= (P_{1/2}(\cdot,y) * (u \zeta) ) (x)$, which is harmonic in $\R^{n+1}_+$ and satisfies $w(\cdot,0) = u\zeta$ in $\R^n$, and set $\psi :=  v- w$.
	Now, on the one hand, since $\psi$ is harmonic in $\B^+_{R+\rho}$ and $\psi =0$ on $B_{R+\rho}$, by standard estimates for harmonic functions (taking for instance the odd extension across $\{y = 0\}$), we have that $\psi \in C^2(\overline{\B^+_R})$.
	On the other hand, since $u\zeta\in C^2(\R^n)$ and has compact support, by the previous results (applied to $u\zeta$ and its derivatives) it follows readily that $w$, $w_y$,  $w_{x_i}$,  $w_{x_i y}$, and $w_{x_i x_j}$ are continuous in $\overline{\B^+_R}$, for every $i,j= 1, \ldots,n$.
	Using that $w$ is harmonic, we deduce that $w_{yy}$ is also continuous in $\overline{\B^+_R}$, concluding the proof.
\end{proof}

%%%%%%%%%%%%%%%%%%%%%%%%%%%
\section{Two simple interpolation inequalities in cubes}
\label{Sec:Interpolation}

The following is an interpolation inequality from \cite{CabreQuantitative}, where it was conceived in order to give quantitative proofs of the Hölder regularity result of \cite{CabreFigalliRosSerra-Dim9} for stable solutions to the local equation $-\Delta u = f(u)$.
For completeness, we present its proof next.
Note that, in contrast with the interpolation inequality of \cite[Theorem~7.28]{GilbargTrudinger}, which requires an extension theorem in $\R^n$, establishing \eqref{5.2} in a cube of~$\R^n$ is immediate once it is proved in dimension one.

\begin{proposition}[\cite{CabreQuantitative}]
	\label{Prop:InterpolationGradient}
	Let $Q=(0,1)^n\subset \R^n$, $p\geq 1$, and $w\in C^{2}(\overline Q)$.
	
	Then, for every $\varepsilon\in (0,1)$,
	\begin{equation}\label{5.2}
		\int_{Q}\abs{\nabla w}^{p}\d x \leq n^{p/2 + 1}\, p\,\varepsilon \int_{Q}\abs{\nabla w}^{p-1}\lvert D^2w\rvert \d x+  n^{p/2 + 1}  \left( \frac{18}{\ep} \right)^{p}\int_{Q}\abs{w}^{p}\d x.
	\end{equation}
\end{proposition}

\begin{proof}
	We follow \cite{CabreQuantitative}.
	We start proving \eqref{5.2} in dimension $n=1$. 
	First, we claim that given $\delta\in (0,1)$ and $w\in C^2([0,\delta])$, it holds
	\begin{equation}\label{5.2claim2}
		\int_{0}^{\delta}\abs{w'}^{p}\d x \leq p\,\delta\int_{0}^\delta\abs{w'}^{p-1}\abs{w''}\d x+9^{p}\delta^{-p}\int_{0}^{\delta}\abs{w}^{p}\d x.
	\end{equation}
	To show this, take $x_0\in[0,\delta]$ such that $\abs{w'(x_0)}= \min_{[0,\delta]}\abs{w'}$.
	Then, for $0<y<\frac\delta3<\frac{2\delta}3<z<\delta$, since $(w(z)-w(y))/(z-y)$ is equal to $w'$ at some point, it follows that $\abs{w'(x_0)}\leq 3\delta^{-1}(|w(y)|+|w(z)|)$. 
	Integrating this inequality first in $y$ and later in $z$, we get that
	$\abs{w'(x_0)}\leq 9\delta^{-2}\int_{0}^{\delta}\abs{w}\d x$, and taking powers we obtain 
	\begin{equation}\label{5.2_1}
		\abs{w'(x_0)}^{p}\leq 9^{p} \delta^{-p - 1}\int_{0}^{\delta}\abs{w}^{p}\d x \quad \text{ for } p\in [1,\infty).
	\end{equation}
	Now, for $x\in(0,\delta)$, we integrate $\left(\abs{w'}^{p}\right)'$ in the interval with end points $x$ and $x_0$, to get
	\[
	\abs{w'(x)}^{p} \leq p\int_{0}^\delta\abs{w'}^{p-1}\abs{w''}\d x+\abs{w'(x_0)}^{p}.
	\] 
	Combining this inequality with \eqref{5.2_1} and integrating in $x\in(0,\delta)$, we conclude \eqref{5.2claim2}.

	Once the previous claim (on integrals in $[0,\delta]$) is proved, we can establish \eqref{5.2} in dimension one.
	Let $w\in C^2([0,1])$. 
	Now, for any given integer $k > 1$ we divide $(0,1)$ into $k$ disjoint intervals of length $\delta=1/k \in (0,1)$. 
	Since \eqref{5.2claim2} did not require any specific boundary values of $w$, we can use the inequality in each of these intervals of length $1/k$ and then add them all up to obtain
	\begin{equation}\label{5.2claim3}
		\int_{0}^1\abs{w'}^{p}\d x \leq  \frac{p}{k}\int_{0}^1\abs{w'}^{p-1}\abs{w''}\d x+(9k)^{p}\int_{0}^1\abs{w}^{p}\d x.
	\end{equation}
	Now, since $0<\ep<1$, in \eqref{5.2claim3} we can choose $k\in \Z$ such that $1 < \frac{1}{\ep}\leq k<\frac{2}{\ep}$, establishing the result in dimension one.
	
	Finally, we show the result for $w\in C^{2}([0,1]^n)$, with $n\geq 2$.
	We will denote $x=(x_1,x')\in\R\times\R^{n-1}$. 
	Using the result in dimension $1$ for every $x'$, we get
	\begin{equation}
		\begin{split}
			\int_{Q}\abs{w_{x_1}}^{p}\d x 
			&=\int_{(0,1)^{n-1}}\d x'\int_0^1\d x_1\abs{w_{x_1}(x)}^{p}
			\\ 
			&\hspace{-1.8cm}\leq p\,\varepsilon\int_{(0,1)^{n-1}}\d x'\int_0^1\d x_1\abs{w_{x_1}(x)}^{p-1}\abs{w_{x_1x_1}(x)} \\
			&+(18\varepsilon^{-1})^{p}\int_{(0,1)^{n-1}}\d x'\int_0^1\d x_1\abs{w(x)}^{p}
			\\
			&\hspace{-1.8cm}= p\,\varepsilon\int_{Q}\abs{w_{x_1}(x)}^{p-1}\abs{w_{x_1x_1}(x)}\d x+ (18\varepsilon^{-1})^{p}\int_{Q}\abs{w(x)}^{p}\d x.
		\end{split}
	\end{equation}
	Since the same inequality holds for the partial derivatives with respect to each variable $x_i$ instead of $x_1$, adding up all the inequalities and using that $|\nabla v|^p \leq n^{p/2} (|v_{x_1}|^p + \ldots + |v_{x_n}|^p)$, we obtain \eqref{5.2}. 
\end{proof}

The following second interpolation inequality is well known and follows immediately from Poincaré's inequality.
It allows us to replace, when $p=2$, the $L^2$ norm of $w$ in \eqref{5.2} by the square of its $L^1$ norm (at the price of adding a reabsorbable small factor of the $H^1$ norm). 

\begin{proposition}
	\label{Prop:InterpolationNash}
	Let $Q=(0,1)^n\subset \R^n$, $p\ge 1$,  and $w\in C^{2}(\overline Q)$.
	
	Then, for every $\tilde\ep\in (0,1)$,
	\begin{equation}\label{5.2bis}
		\int_{Q}|w|^{p}\d x \leq C\, \tilde\ep^{\;p} \int_{Q}\abs{\nabla w}^{p}\d x+  C  \, {\tilde\ep}^{\; -n(p-1)}\left( \int_{Q}\abs{w}\d x\right)^p
	\end{equation}
	for some constant $C$ depending only on $n$ and $p$.
\end{proposition}

\begin{proof}
	From Poincar\'e's inequality for functions $w$ in a cube $Q_\delta$ of side-length $\delta$, which reads $\|w-w_{Q_\delta}\|_{L^p (Q_\delta)} \leq C\delta\|\nabla w\|_{L^p (Q_\delta)}$ (where $w_Q$ is the average of $w$ in $Q$), we obtain 
	$$
	\|w\|_{L^p (Q_\delta)} \leq C\delta \|\nabla w\|_{L^p (Q_\delta)}+ |Q_\delta|^{1/p}|w_{Q_\delta}|= C\delta \|\nabla w\|_{L^p (Q_\delta)}+ \delta^{-n(p-1)/p}\int_{Q_\delta}\abs{w}\d x.
	$$
	Now, for any given integer $k>1$ we divide $Q$ into $k^n$ disjoint cubes $Q_j$ of side-length $\delta=1/k$. 
	Using the previous inequality in each cube $Q_j$ we have
	$$
	\int_{Q_j}|w|^{p} \d x \leq \dfrac{C}{k^{p}} \int_{Q_j}\abs{\nabla w}^{p}\d x+  Ck^{n(p-1)}\left( \int_{Q_j}\abs{w}\d x\right)^p.
	$$
	Adding up all these inequalities and noticing that $\sum_j (\int_{Q_j}\abs{w}\d x)^p \le ( \sum_j \int_{Q_j}\abs{w}\d x)^p= ( \int_{Q}\abs{w}\d x)^p$, the desired estimate \eqref{5.2bis} follows after choosing an integer $k$ such that $1 <\tilde\ep^{\,-1}\le k\le 2 \tilde\ep^{\,-1}$.
\end{proof}

%%%%%%%%%%%%%%%%%%%%%%%%%%%%%%%%%%%%%%%%
\section*{Acknowledgements}
%%%%%%%%%%%%%%%%%%%%%%%%%%%%%%%%%%%%%%%%
We thank Louis Dupaigne for suggesting to add the Liouville result of \Cref{Coro:Liouville}.

%%%%%%%%%%%%%%%%%%%%%%%%%%%%%%%%%%%%%%%%%%%%%%%%%%%%%%%%%%%%%%%%%%%%%%%%%%%%
%%%%%%%%%%%%%%%%%%%%%%%%%%%%%%%%%%%%%%%%%%%%%%%%%%%%%%%%%%%%%%%%%%%%%%%%%%%%

\end{document}